\documentclass[11pt]{amsart}
\usepackage{amsmath,amsfonts,amssymb,amsthm}
\usepackage[alphabetic]{amsrefs}
\usepackage{hyperref}
\usepackage{graphicx,color}
\usepackage{pictexwd,dcpic,epsf}

\newtheorem{theorem}{Theorem}[section]
\newtheorem{lemma}[theorem]{Lemma}
\newtheorem{prop}[theorem]{Proposition}
\newtheorem*{mtheorem}{Theorem 2}
\newtheorem*{theorem1}{Theorem 1}
\newtheorem*{theorem3}{Theorem 4}

\newtheorem{corollary}[theorem]{Corollary}
\newtheorem*{T3}{Theorem 3}
\theoremstyle{definition} 
\newtheorem{definition}{Definition}[section]

\newtheorem{remark}[theorem]{Remark}

\theoremstyle{remark} 

\newcommand{\ov}{\overline}
\newcommand{\wt}{\widetilde}

\newcommand {\xz}{X^{(0)}}
\newcommand {\xj}{X^{(1)}}
\newcommand {\dw}{d_{\mathcal W}}

\newcommand {\ag}{A(\gamma)}
\newcommand {\gi}{\mr{girth} \,}
\newcommand {\di}{\mr{diam} \,}
\newcommand {\bds}{($\beta,\Phi$)--se\-pa\-ration }
\newcommand {\mr}{\mathrm}
\newcommand {\cA}{\mathcal A}

\newcommand {\tl}{l}

\newcommand {\txT}{\Theta}

\newcommand {\stxT}{\Theta}
\newcommand {\stxm}{m}

\newcommand {\txxT}{\Theta}
\newcommand {\txxml}{m}
\newcommand {\hxxT}{\wt {\Theta}}
\newcommand {\hxxm}{\wt m}
\newcommand {\stxxT}{\widehat{\Theta}}
\newcommand {\stxxm}{\widehat{m}}
\newcommand {\tn}{\Theta_n}

\newcommand {\bN}{\mathbb N}

\begin{document}

\title[Small cancellation labellings of graphs and applications]{Small cancellation labellings of some infinite graphs and applications}

\author{Damian Osajda}
\address{Instytut Matematyczny,
Uniwersytet Wroc\l awski\\
pl.\ Grunwaldzki 2/4,
50--384 Wroc{\l}aw, Poland}
\address{Universit\"at Wien, Fakult\"at f\"ur Mathematik\\
Oskar-Morgenstern-Platz 1, 1090 Wien, Austria.
}
\email{dosaj@math.uni.wroc.pl}


\begin{abstract}
We construct small cancellation labellings for some infinite se\-quen\-ces of finite graphs of bounded degree.  
We use them to define infinite graphical small cancellation presentations of groups.
This technique allows us to
provide examples of groups with exotic properties:

$\bullet$ We construct the first examples of finitely generated coarsely non-amenable groups (that is, groups without Guoliang Yu's Property A) that 
are coarsely embeddable into a Hilbert space. Moreover, our groups act properly on CAT(0) cubical complexes.

$\bullet$ We construct the first examples of finitely generated groups, with expanders embedded isometrically into their Cayley graphs -- in contrast, in the case of the Gromov monster expanders are not even coarsely embedded.

We present further applications. 
\end{abstract}

\maketitle

\section{Introduction}
\label{s:intro}
The main goal of this article is to present a technique of constructing finitely generated groups such that given (infinite) graphs embed isometrically into their Cayley graphs.
This allows one to obtain groups with some features resembling the ones of those graphs. In particular, we construct groups without Guoliang Yu's property A that are coarsely embeddable into a Hilbert space (see Subsection~\ref{s:inA} in this Introduction below), and we construct groups, into whose Cayley graphs some expanders embed isometrically (see Subsection~\ref{s:inGr}). The latter groups are therefore not coarsely embeddable into Hilbert spaces, and various versions of the Baum-Connes conjecture fail for them. The general tool we use is the graphical small cancellation theory, and the main technical point is then finding appropriate small cancellation labellings of the graphs in question (see the next Subsection~\ref{s:inG}).

\subsection{Small cancellation labellings of some graphs}
\label{s:inG}
A labelling of a graph may be seen as an assignment of labels to directed edges; see details in Section~\ref{s:constr}. 
A labelling satisfies some small cancellation condition when no labelling of a long path (long with respect to the girth) appears in two different places; see Subsection~\ref{s:fscl}.
For our purposes we are interested in a finite set of labels, and in graphs being infinite disjoint unions of finite graphs with degree bounded uniformly. Examples are sequences of finite $D$--regular graphs, for a fixed degree $D>2$. For such graphs the only `small cancellation' labelling
provided till now was the famous Gromov labelling of some expanders \cite{Gro} (cf.\ some explanations of this construction in \cites{AD,Cou}). 
Gromov's labelling is in a sense generic, and as such cannot satisfy the small cancellation condition we work with (see the discussion in Subsection~\ref{s:discus}). Therefore Gromov's labelling defines a weak embedding in the sense of \cite[Definition 7.2]{Ost2013}, but not a coarse embedding of the graphs (relators) into the corresponding group (see Subsection~\ref{s:discus} for details). (Recall that a map $f\colon (X,d_X) \to (Y,d_Y)$ between metric spaces is a \emph{coarse embedding} when $d_Y(f(x_n),f(y_n))\to \infty$ iff $d_X(x_n,y_n)\to \infty$ for all sequences $(x_n)$, $(y_n)$.)
We study sequences $(\Theta_n)_{n\in \bN}$ of finite graphs of uniformly bounded degree, with growing girth, and diameters bounded in terms of girth (see Section~\ref{s:constr} for details). For them, we construct labellings satisfying much more restrictive conditions then the Gromov labellings do.

\begin{theorem1}[see Theorem~\ref{l:c2LLL} in the text]
For every $\lambda >0$ there exists a $C'(\lambda)$--small cancellation labelling of $(\Theta_n)_{n\in \bN}$ over a finite set of labels. 
\end{theorem1}

It is well known (see e.g.\ \cite{Gro,Oll}) that satisfying such strong small cancellation condition implies that for groups that we construct using this labelling, 
the graphs $\tn$ are isometrically embedded into the Cayley graphs.

For constructing the desired labellings we use techniques coming from combinatorics (graph colorings) \cite{AloGry} and relying on the Lov\' asz Local Lemma (see e.g.\ \cite{AloSpe}). This is a novelty in the subject.
Note that whereas the core of our method is probabilistic (similarly as Gromov's techniques), there is a fundamental difference with Gromov's approach: We look for any labelling with required properties, while in the other method the properties of the generic labelling are explored. This is crucial for getting stronger features, as explained in Subsection~\ref{s:discus}. 
The tools used in both approaches are different. Our argument is also relatively short (pp.\ 5--14 below) compared to Gromov's one as presented in \cite{AD}. 

Below we describe the actual applications of the small cancellation labellings we construct. Nevertheless, we believe that the construction itself, and the overall combinatorial technique developed in this article, are important tools that will find many applications beyond the scope 
presented here. 

\subsection{Non-exact groups with the Haagerup property}
\label{s:inA}
{Property A}, or \emph{coarse amenability}, was introduced by Guoliang Yu \cite{Yu} for his studies on the Baum-Connes conjecture.
A uniformly discrete metric space $(X,d)$ has \emph{property A} if for every $\epsilon >0$ and $R>0$ there exist a collection of finite subsets
$\{ A_x \} _{x\in X}$, $A_x\subseteq X\times \bN$ for every $x\in X$, and a constant $S>0$ such that 
\begin{enumerate}
\item
$\frac{|A_x \triangle A_y |}{|A_x\cap A_y|}\leqslant \epsilon$ when $d(x,y)\leqslant R$, and
\item
$A_x\subseteq B(x,S)\times \bN$.
\end{enumerate}
A finitely generated group has property A if it is coarsely amenable for the word metric with respect to some
finite generating set.

Property A may be seen as a weak (non-equivariant) version of amenability, and similarly to the latter notion it has many equivalent formulations and 
a large number of significant applications; see e.g.\ \cites{Willett-notes,NowakYu}. 
For countable discrete groups, Property A is equivalent to: the existence of a topological amenable action on a compact Hausdorff space \cite{HigRoe}, to the exactness of the reduced $C^{\ast}$--algebra \cites{GK,Oz}, to nuclearity of the uniform Roe algebra \cite{Roe},  
and to few other geometric and analytic properties; see e.g.\ \cite[pp.\ 81--82]{NowakYu}.  

Property A implies coarse embeddability into a Hilbert space \cite{Yu}. Analogously, amenability implies the Haagerup property (that is, a-T-me\-na\-bility in the sense of Gromov). The following diagram depicts relations (arrows denoting implications) between those properties for groups; see e.g.\ \cite[p.\ 124]{NowakYu}. Observe that the notions on the right may be seen as non-equivariant counterparts of the ones on the left.

$$  \begindc{\commdiag}[20]
\obj(1,1)[a]{Haagerup property}
\obj(77,1)[b]{coarse embeddability into $l_2$}
\obj(1,28)[c]{amenability}
\obj(77,28)[d]{property A}
\mor{a}{b}{}
\mor{c}{d}{}
\mor{c}{a}{}
\mor{d}{b}{}
\enddc  $$

\medskip
In view of the above a natural question, which was open till now, arose: \emph{Do groups coarsely embeddable into a Hilbert space have property A?} -- see e.g.\ 
\cite[Remark 3.8(2)]{claire}, \cite[Problem 3.4]{GH}, \cite[p.\ 257 \& 261]{GK2}, \cite[p.\ 6]{NiSaW}, \cite[footnote p.\ 27]{AD}, \cite[p.\ 251]{Willett-notes}, or \cite[Open Question 5.3.3]{NowakYu}. 
Approaches to answer this question (also in the positive) attracted much research in the area and triggered many new ideas.
Following a program towards a negative answer initiated in \cite{AO}, we prove a stronger statement.
\begin{mtheorem}[see Theorem~\ref{t:main} in the text]
There exist finitely generated groups acting properly on CAT(0) cubical complexes and not having property A.
\end{mtheorem}
Acting properly on a CAT(0) cubical complex is equivalent to acting properly on a space with walls \cites{HP,Nica,ChaNib}, that is to having property PW (in a language of \cite{Cor}). This implies in particular the Haagerup property, and hence equivariant coarse embeddability into a Hilbert space.
Theorem 2 shows that the diagram above is complete -- there are no other implications between the properties there; see \cite[p.\ 124]{NowakYu}.
Besides the Gromov monsters \cite{Gro}, the groups constructed in the current paper (see also Subsection~\ref{s:inGr} below) are the only finitely generated groups without property A known at the moment; see e.g.\ \cite{Nowak}, \cite[p.\ 6]{NiSaW}, \cite[p.\ 28]{AD}, \cite[p.\ 251 and Section 7.5]{Willett-notes}, or \cite[Open Question 4.5.4]{NowakYu} for related remarks and questions. 
 Note that coarsely non-amenable spaces embeddable into $l_2$ were constructed in \cite{Nowak} (locally finite case) and in \cite{AGS} (bounded geometry case).  Our construction relies on examples constructed in \cite{Ost}.

Let us remark here that the lack of property A for a group was believed to be an essential obstacle to various Baum-Connes conjectures by some experts. This question is clarified by Theorem 2: There are groups without property A but satisfying the Haagerup property. For such groups the strong Baum-Connes conjecture holds \cite{HK}.

Coarsely non-amenable groups embeddable into a Hilbert space construc\-ted in this article are given by infinite graphical small cancellation presentations (see Section~\ref{s:group} for details). The infinite family of graphs being relators consists of some coverings
of regular graphs with girths growing to infinity. Relators are graphs with walls (see Section~\ref{s:wall}), and thus there is a walling for the 
group itself (see the proof of Theorem~\ref{t:main}). Therefore, the
group acts on a space with walls. This action is proper if some additional conditions are satisfied. We study such a condition -- the proper lacunary walling condition -- in Section~\ref{s:wlac}. This is a theory of independent interest that relies on, and extends in a way, the preceding work of the author with Goulnara Arzhantseva \cite{AO} (cf.\ also \cite{AO0}). In particular, we obtain the following analogue of
\cite[Main Theorem and Theorem 1.1]{AO}.
\begin{T3}[see Theorem~\ref{p:lsp} in the text]
Let $X$ be a complex satisfying the proper lacunary walling condition. Then the wall pseudo-metric is proper. Consequently, a group acting
properly on $X$ acts properly on a CAT(0) cubical complex.
\end{T3}
A group as in Theorem 2 is constructed so that the proper lacunary walling condition is satisfied for a space acted properly upon by the group. 
Therefore the group acts properly on a CAT(0) cubical complex. On the other hand, by the small cancellation condition, the infinite family of relators embeds isometrically into the Cayley graph. Since, by a result of Willett~\cite{Willett}, such a family has not property A, we conclude that the whole group is coarsely non-amenable.

\subsection{Groups with expanders in Cayley graphs}
\label{s:inGr}
Using his labelling of expanders Gromov constructed a finitely generated group, for which there exists  a weak embedding in the sense of \cite[Definition 7.2]{Ost2013}
of an expander \cite{Gro}. A weak embedding is not necessarily a coarse embedding and with Gromov's construction one cannot obtain
the latter; see the discussion in Subsection~\ref{s:discus}. Having weakly embedded expanders is enough to claim that the group does not coarsely embed into a Hilbert space \cite{Gro}, or that the Baum-Connes conjecture with coefficients fails for such groups \cite{HiLaSka} (cf.\ our Corollary~\ref{c:Hilbert} and Corollary~\ref{c:WiYu}). However, in many other situations
it seems to be necessary to have an actual coarse embedding of an expander to obtain desired properties; see e.g.\ \cites{WillettYu1}.
Our labelling allows us to provide groups with such a property and more, as the following result shows.

\begin{theorem3}[see Corollary~\ref{c:Hilbert} in the text]
There exist finitely generated groups with expanders isometrically embedded into their Cayley graphs.
\end{theorem3}

The existence of such examples is crucial for some analyses of failures of the Baum-Connes conjecture with coefficients, as in \cite[Theorem 8.3]{WillettYu1} (see Corollary~\ref{c:WiYu}) or in \cite[Section 7]{BGW}. Besides Gromov's monsters (and groups derived from them), our examples are the only finitely generated counterexamples to the Baum-Connes conjecture with coefficients, and the only finitely generated groups not coarsely
embeddable into Hilbert space, known at the moment.

As direct consequence of Theorem 4 and a result by Sapir~\cite{Sapir} we obtain that there exist closed aspherical manifolds whose fundamental groups contain coarsely embedded expanders; see Corollary~\ref{c:Sapir}. Those are the first examples of this type. 
%

Note that in some situations it may be necessary to have the actual isometric embedding of given graphs into groups -- this happens for example in our construction of PW non-A groups; see Subsection~\ref{s:inA} above and Section~\ref{s:pwna}. There we need it for the delicate construction of walls. We believe that it may be crucial
for further applications.
\medskip

\noindent
{\bf Acknowledgments.} First and foremost, I would like to thank Goulnara Arzhantseva for introducing me to the subject, leading through it, and for the great collaboration preceding this work.
I am grateful for encouragement, for helpful discussions, and/or for remarks improving the manuscript to Dominik Gruber, Vincent Guirardel, Fr\' ed\' eric Haglund, Ashot Minasyan, Piotr Nowak, Denis Osin, Mark Sapir, J\' an \v Spakula, and Rufus Willett. I thank the anonymous referee for a careful reading of the manuscript and
numerous important comments.

This research was partially supported by Narodowe Centrum Nauki, grants no.\ UMO-2012/06/A/ST1/00259,
UMO-2015/\-18/\-M/\-ST1/\-00050, and UMO-2017/\-25/\-B/\-ST1/\-01335, and by the ERC grant ANALYTIC no.\ 259527.

\section{Small cancellation labellings of some graphs}
\label{s:constr}

The goal of this section is proving Theorem 1 from Introduction or, more precisely, Theorem~\ref{l:c2LLL} below.
Considering a metric on a graph we always mean a metric on the set of vertices, being a path metric within connected components. 
\medskip

\emph{Throughout this paper we work with the sequence $\Theta=(\tn)_{n\in \bN}$ of disjoint finite connected graphs of degree bounded by $D>0$. Furthermore, we have $\gi \tn \stackrel{n\to \infty}{\longrightarrow} \infty$ and $\Theta$ satisfies the following condition:
\begin{align}
\label{e:dig}
\di \tn \leqslant A\, \gi \tn,
\end{align}
where $\di$ denotes the diameter, $\gi$ is the length of the shortest simple cycle, and $A$ is a universal (not depending on $n$) constant.
For this section we fix a \emph{small cancellation constant} $\lambda \in (0,1/6]$. We also assume that $1<\lfloor \lambda \gi \Theta_{n}\rfloor<\lfloor \lambda \gi \Theta_{n+1}\rfloor$.}

\medskip

Observe that for a sequence $(\tn)_{n\in \bN}$ with growing girths, the last assumption can be fulfilled by passing to a subsequence -- this is allowed from the point of view of our applications. 

By a \emph{labelling} $(\Gamma,f)$ of an undirected graph $\Gamma$ we mean a graph morphism $f\colon \Gamma \to W$ into a bouquet of finitely many loops $W$, that is a graph with one vertex end several edges. Usually we refer however to the following interpretation of 
the labelling $f$. Orient edges of $W$ and decorate every directed edge (loop) by an element of a finite set $S$. Then the labelling $f$ is determined by the following data: We orient every edge of $\Gamma$ and we assign to it the corresponding element of the set $S$ or \begin{figure}[h!]
\centering
\includegraphics[width=0.6\textwidth]{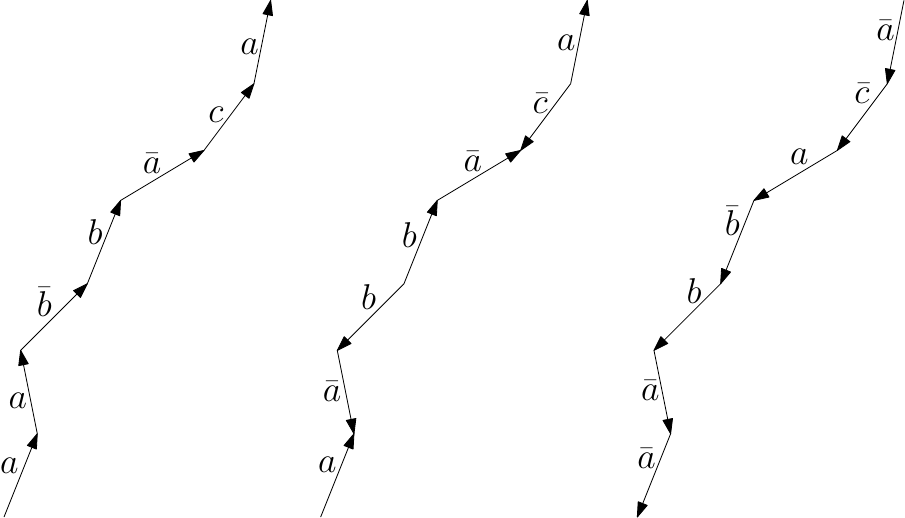}
\caption{Three representations of the same labelling.}
\label{f:label}
\end{figure}
an element of the set $\ov S$ of formal inverses of elements of $S$. We call the set $S\cup \ov S$ the \emph{(symmetrized) set of labels}, and by $\bar s$ we denote the \emph{inverse} of $s$. Using this interpretation we identify the labelling assigning the label $s$ to an oriented edge
$vw$ with the labelling of $wv$ by $\bar s$; see Figure~\ref{f:label}.
The labelling $(\Gamma,f)$ is \emph{reduced} if $f\colon \Gamma \to W$ is locally injective, that is, 
for every vertex and every two edges leaving the vertex their labels are different.
We will usually not specify the (symmetrized) set of labels (although it will change often) -- we will just mention that it is finite. 

We construct the small cancellation labelling $(\Theta,m)=((\tn,m_n))_{n\in \bN}$ in three steps. First, in Subsection~\ref{s:LLL}
we construct a labelling $(\Theta,l)=((\tn,l_n))_{n\in \bN}$ such that $l_n$--labellings of long (relative to $\gi \tn$) paths in $\tn$ do not appear in $(\Theta_{n'},l_{n'})$, for $n\neq n'$; see Lemma~\ref{l:sc1}.
Then, in Subsection~\ref{s:mLLL} we construct a labelling $(\Theta,l')=((\tn,l'_n))_{n\in \bN}$ with the property that, for each $n$, long paths in $\tn$ are labelled differently; see Lemma~\ref{l:scm}. Finally, in Subsection~\ref{s:fscl} we combine $(\Theta,l)$ and $(\Theta,l')$ to obtain the required small cancellation labelling $(\Theta,m)$; see Theorem~\ref{l:c2LLL}. 


\subsection{The labelling \texorpdfstring{$(\Theta,l)$}{(Theta,l)}: small cancellation between graphs}
\label{s:LLL}

Recall the following version of the Lov\'asz Local Lemma (see e.g.\ \cite{AloSpe}) that can be found in
\cite[Lemma 1]{AloGry}. Here $\mr{Pr}(A)$ denotes the (discrete) probability of an event $A$, and $\bar A$ denotes
the opposite event (complementary set).  

\begin{lemma}[Lov\'asz Local Lemma]
\label{l:LLL}
Let $\cA = \cA_1 \cup \cA_2 \cup \ldots \cup \cA_r$ be a
partition of a finite set of events $\cA$, with $\mr{Pr}(A) = p_i$ for every $A \in \cA_i$, $i = 1, 2,\ldots , r$.
Suppose that there are real numbers $0\leqslant a_1, a_2, \ldots, a_r < 1$ and $\Delta_{ij}\geqslant 0$, $i, j = 1, 2, \ldots, r$
such that the following conditions hold:
\begin{enumerate}
\item[\it (i)]
for any event $A \in \cA_i$ there exists a set $\mathcal D_A \subseteq \cA$ with $|\mathcal D_A \cap \cA_j|\leqslant \Delta_{ij}$ for all
$j = 1, 2,\ldots, r$, such that $A$ is independent of $\cA \setminus (\mathcal D_A \cup\{ A\})$,
\item[{\it (ii)}]
$p_i\leqslant a_i \prod_{j=1}^{r} (1-a_j)^{\Delta_{ij}}$ for all $i = 1, 2,\ldots, r$.
\end{enumerate}
Then $\mr{Pr}(\bigcap_{A\in \cA} \bar {A})>0$.
\end{lemma}

Let $\gamma_n=\lfloor \lambda \, \gi \txT_n \rfloor$. Observe that $\lambda \, \gi \txT_n -1 < \gamma_n$ and thus
\begin{align}
\label{e:LLL8}
\frac{\gi \txT_n}{\gamma_n}<\frac{1}{\lambda}+\frac{1}{\lambda\, \gamma_n}<\frac{2}{\lambda}.
\end{align}
We will find a labelling $(\txT, \tl)=((\txT_n,\tl_n))_{n\in \bN}$ with $L$ labels such that 
$\tl_n$--labellings of paths of length at least $\gamma_n$ do not appear as $\tl_{n'}$--labellings, for $n'>n$.
Unless stated otherwise, we always assume that paths are without backtracking. It implies that all paths shorter than the girth are simple.
Define $L$ as follows (here $e$ denotes the Euler constant):
\begin{align}
\label{f:LLL9}
L:={\left \lceil {2De^4D^{\frac{2A}{\lambda}+1}}\right \rceil}.
\end{align}
The number $e_n$ of edges of $\Theta_n$ is bounded by $e_n\leqslant D^{\mr{diam}\, \Theta_n}$. Thus, by the condition
(\ref{e:dig}), we have 
\begin{align}
\label{e:LLL9a}
e_n\leqslant D^{A\gi \Theta_n}.
\end{align}

We construct $((\txT_n, \tl_n))_{n\in \bN}$ inductively: $(\txT_1,\tl_1)$ is an arbitrary labelling with $L$ labels, and further we execute an inductive step.
Assume that $(\txT_1,\tl_1),\ldots,(\txT_{n-1},\tl_{n-1})$ are defined. 
Let $M_i$ denote the number of words appearing as labels of paths of length $\gamma_i$ in $(\Theta_i,l_i)$.
Let $N_i$ denote the number of possibilities of labelling a fixed simple path of length $\gamma_i$ by $L$ letters.
Observe that, for $i=1,\ldots,n-1$, we have 
\begin{align}
\label{e:LLL11}
M_i< e_iD^{\gamma_i},
\end{align}
and
\begin{align}
\label{e:LLL12}
N_i= L^{\gamma_i}.
\end{align}

The labelling $(\txT_n,\tl_n)$ is then one given by the following lemma.
\begin{lemma}
\label{l:scLLL}
There exists a labelling $(\txT_n,\tl_n)$ with $L$ labels such that, for $i=1,2,\ldots,n-1$, no $\tl_i$--labelling of a path of length
$\gamma_i$ in $\txT_i$ appears as an $\tl_n$--labelling of a path of length
$\gamma_i$ in $\txT_n$.
\end{lemma}
\begin{proof}
We use the Lov\'asz Local Lemma \ref{l:LLL} following closely the proof of \cite[Theorem 1]{AloGry}. 
Randomly label the edges of $\txT_n$ by $L$ labels. 
For a path $p$ in $\txT_n$ of length $\gamma_i$, let $A(p)$ denote the event that its $\tl_n$--labelling is the same as an $\tl_i$--labelling of some path in $\txT_i$ of length $\gamma_i$, for $i<n$.
Set $\cA_i=\{ A(p) :  p \; \mr{is\; a\; path\; of\; length}\; \gamma_i \;\mr{in} \; \txT_n \}$. Recall (see Lemma \ref{l:LLL}) that $p_i$ denotes the probability $\mr{Pr}(A)$ for every $A\in \cA_i$. Then, by (\ref{e:LLL11}), (\ref{e:LLL12}), (\ref{e:LLL9a}), and (\ref{e:LLL8}), we have
\begin{align}
\label{e:LLL15}
p_i\leqslant \frac{e_iD^{\gamma_i}}{L^{\gamma_i}}\leqslant \frac{D^{A\gi \Theta_i+\gamma_i}}{L^{\gamma_i}}=
\left(\frac{D^{\frac{A\gi \Theta_i}{\gamma_i}+1}}{L}\right)^{\gamma_i}<
\left(\frac{D^{\frac{2A}{\lambda}+1}}{L}\right)^{\gamma_i}.
\end{align}
 Each path of length $\gamma_i$ shares an edge with not more than $\gamma_i\gamma_jD^{\gamma_j}$ paths of length $\gamma_j$, so that we may take 
$\Delta_{ij}=\gamma_i\gamma_jD^{\gamma_j}$.
Let $a_i=a^{-\gamma_i}$, where $a=2D$. 
Then, by using subsequently: formulas (\ref{e:LLL15}) and (\ref{f:LLL9}), the definition of $a_i$, the fact that $\sum_{j=1}^{\infty} j/2^j =2$, the definitions of $a$, $\Delta_{ij}$, and $a_j$, we obtain:

\begin{align*}
\begin{split}
p_i&< \left(\frac{D^{\frac{2A}{\lambda}+1}}{L}\right)^{\gamma_i}\leqslant 2^{-\gamma_i}D^{-\gamma_i}e^{-4\gamma_i}
= a_i \exp{\left(-2\sum_{j=1}^{\infty}\gamma_i\, \frac{j}{2^j} \right)}
\\ & \leqslant 
a_i \exp{\left(-2\sum_{j}\gamma_i\, \frac{\gamma_j}{2^{\gamma_j}} \right)}
 =a_i \exp{\left(-2\sum_{j}\gamma_i\gamma_j\left(\frac{D}{a}\right)^{\gamma_j} \right)}
\\ & =a_i \exp{\left(-2\sum_{j} \Delta_{ij}a_j \right)}
= a_i\prod_j e^{-2a_j\Delta_{ij}}.
\end{split}
\end{align*}
Since, by $a_j\leqslant 1/2$, we have $e^{-2a_j}\leqslant (1-a_j)$ (because for the function $f\colon \mathbb R \to \mathbb R\colon x \mapsto e^{-2x}$ we have $f(0)=1-0$, $f(\frac{1}{2})<1-\frac{1}{2}$, and $f'$ is increasing), we obtain finally
\begin{align*}
p_i\leqslant a_i\prod_j (1-a_j)^{\Delta_{ij}}.
\end{align*}
Therefore the hypotheses of the Lov\'asz Local Lemma are fulfilled, and we conclude that there exists a labelling $\tl_n$ as required.
\end{proof}

The labelling $(\Theta,l)=((\tn,l_n))_{n\in \bN}$ with $L$ labels obtained by the inductive construction has the following property.

\begin{lemma}
\label{l:sc1}
For each $n\in \bN$, no $l_n$--labelling of a path of length at least $\lambda \, \gi \tn$ is a labelling of a path in $(\Theta_{n'},l_{n'})$, with
$n'\neq n$.
\end{lemma}


\subsection{The labelling \texorpdfstring{$(\Theta,l')$}{(Theta, l')}: small cancellation within \texorpdfstring{$\tn$}{Theta n}}
\label{s:mLLL}
For this subsection we fix $n$ -- we will work only with $\tn$. Again, unless stated otherwise, we always assume that paths are without backtracking, in particular all paths shorter than the girth are simple.
First we show that if two distinct relatively long paths in $\tn$ have the same  
labelling then a path with a specific labelling appears; see Lemma~\ref{l:mLLL}. Then we use this observation to find a required labelling $(\tn,l'_n)$, by an application of the Lov\' asz Local Lemma, similarly as in the proof of Lemma~\ref{l:scLLL}.
\medskip 

Let $\widetilde v=(v_0,v_1,\ldots,v_k)$, $\widetilde w=(w_0,w_1,\ldots,w_k)$ be two paths with the same labelling and with $k=\lfloor \lambda \, \gi \tn \rfloor$
(here $v_i,w_i$ are consecutive vertices). 
Denote the labelling of the directed edge $v_{i-1}v_{i}$ by $a_i$, for $i=1,2,\ldots,k$.
We consider separately the cases when $\widetilde v$ and $\widetilde w$ share an edge, and when they do not.
\medskip

\noindent
{\bf{Case I: $\widetilde v$ and $\widetilde w$ do not share an edge.}} Then there exists a path $\widetilde u = (u_0:=v_s,u_1,\ldots,u_r:=w_t)$ of minimal length
connecting $\widetilde v$ and $\widetilde w$. Possibly $r=0$, that is, $\widetilde u$ is one vertex $u_0:=v_s=w_t$. Without loss of generality (subject to renaming) we may assume that $s\geqslant t \geqslant k/2$ (if $s<t$ we may exchange $\widetilde v$ with $\widetilde w$, if $t<k/2$ then we exchange $w_i$ with $w_{k-i}$ -- this corresponds to difference in labellings in Cases Ia and Ib below); see Figure~\ref{f:I}. By our assumptions we have $r\leqslant \di \tn \leqslant A\, \gi \tn$. 
We consider the following two cases separately.
\begin{figure}[h!]
\centering
\includegraphics[width=0.6\textwidth]{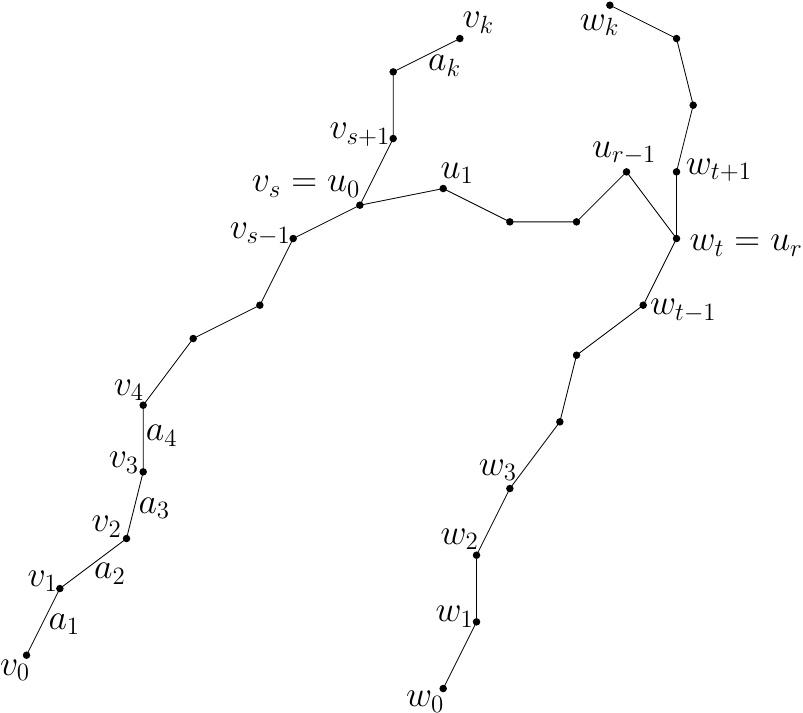}
\caption{Case I}
\label{f:I}
\end{figure}

\medskip

\noindent
\emph{(Case Ia): The labelling of a directed edge $w_{i-1}w_i$ is $a_i$} (see Figure~\ref{f:Ia} on the left).
Then we have the path $p:=(v_0,\ldots,v_s,u_1,\ldots,u_{r-1},w_t,\ldots,w_0)$.
By (\ref{e:dig}), its length $|p|$ may be bounded from above by
\begin{align}
\label{e:m1}
2k+r\leqslant 2\lambda \, \gi \tn +A\, \gi \tn=(2\lambda+A)\gi \tn.
\end{align}
In its labelling the beginning sub-path of length $t$ is labelled the same way -- up to changing orientation -- as the ending
sub-path of length $t$, that is, it has the form (where `repetitive' parts are underlined):
\begin{align}
\label{e:m2}
(\underline{a_1,a_2,\ldots,a_t},\ldots,\underline{\bar a_t,\ldots,\bar a_2,\bar a_1}),
\end{align}
with 
\begin{align}
\label{e:m3}
t\geqslant k/2>\frac{\lambda \gi \tn}{4}.
\end{align}
(The last inequality is a rough estimate coming from $k> \lambda \gi \tn -1$.)

\begin{figure}[h!]
\centering
\includegraphics[width=0.9\textwidth]{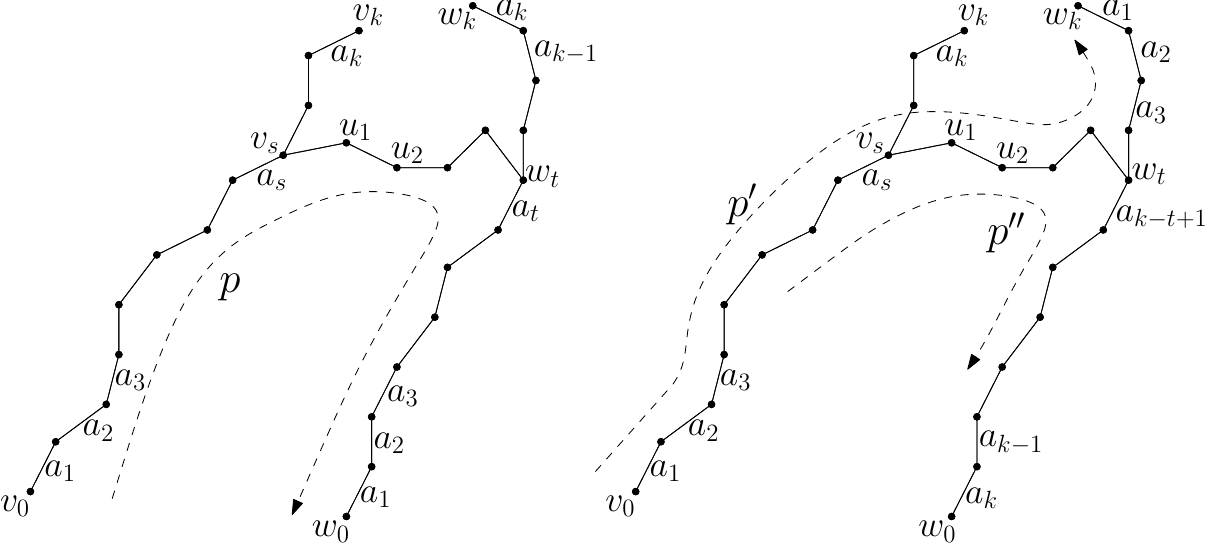}
\caption{Case Ia (left) and Case Ib (right)}
\label{f:Ia}
\end{figure}

\medskip

\noindent
\emph{(Case Ib): The labelling of a directed edge $w_{i+1}w_{i}$ is $a_{k-i}$} (see Figure~\ref{f:Ia} on the right).
In this case again we consider separately two subcases:
\medskip

(i) When $t\leqslant 3k/4$ then we consider the path $p':=(v_0,\ldots,v_s,u_1, \ldots,u_{r-1},\\ w_t,\ldots,w_k)$.
Its length may be again bounded from above by (\ref{e:m1}), and its labelling is of the form similar to (\ref{e:m2}):
\begin{align*}
(\underline{a_1,a_2,\ldots,a_{k-t}},\ldots,\underline{\bar a_{k-t},\ldots,\bar a_2,\bar a_1}),
\end{align*}
with 
\begin{align}
\label{e:m5}
k-t\geqslant k-3k/4=k/4>\frac{\lambda \gi \tn}{8}.
\end{align}
\medskip

(ii) When $t> 3k/4$ then we consider the path $p'':=(v_{k-t},\ldots,v_s,u_1, \ldots,\\ u_{r-1},w_t,\ldots,w_{k-s})$.
We bound its length from above by (\ref{e:m1}), and its labelling is of the form:
\begin{align*}
(\underline{a_{k-t+1},a_{k-t+2},\ldots,a_s},\ldots,\underline{a_{k-t+1},a_{k-t+2},\ldots,a_s}),
\end{align*}
with the lengths of the `repetitive' pieces at least: 
\begin{align}
\label{e:m7}
s-(k-t+1)+1=s+t-k> \frac{k}{2}+\frac{3k}{4}-k=k/4>\frac{\lambda \gi \tn}{8}.
\end{align}

\medskip

\noindent
{\bf{Case II: $\widetilde v$ shares an edge with $\widetilde w$.}}
Then there are $r\geqslant 1$, and $s,t$, such that $v_{s+i}=w_{t+i}$, for $i=1,2,\ldots,r$, and $v_i\neq w_j$ in other cases (because the paths are much shorter than the girth). 
Similarly as in Case I, without loss of generality (subject to renaming) we may assume that $s\geqslant t$; see Figure~\ref{f:II}. 
We consider the following two cases separately.
\begin{figure}[h!]
\centering
\includegraphics[width=0.5\textwidth]{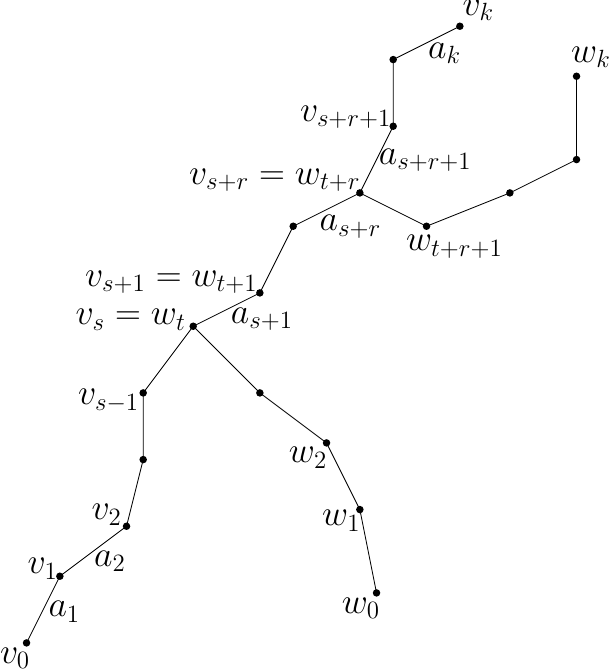}
\caption{Case II}
\label{f:II}
\end{figure}
\medskip

\noindent
\emph{(Case IIa): The labelling of a directed edge $w_{i-1}w_i$ is $a_i$} (see Figure~\ref{f:IIb} on the left).
\begin{figure}[h!]
\centering
\includegraphics[width=0.8\textwidth]{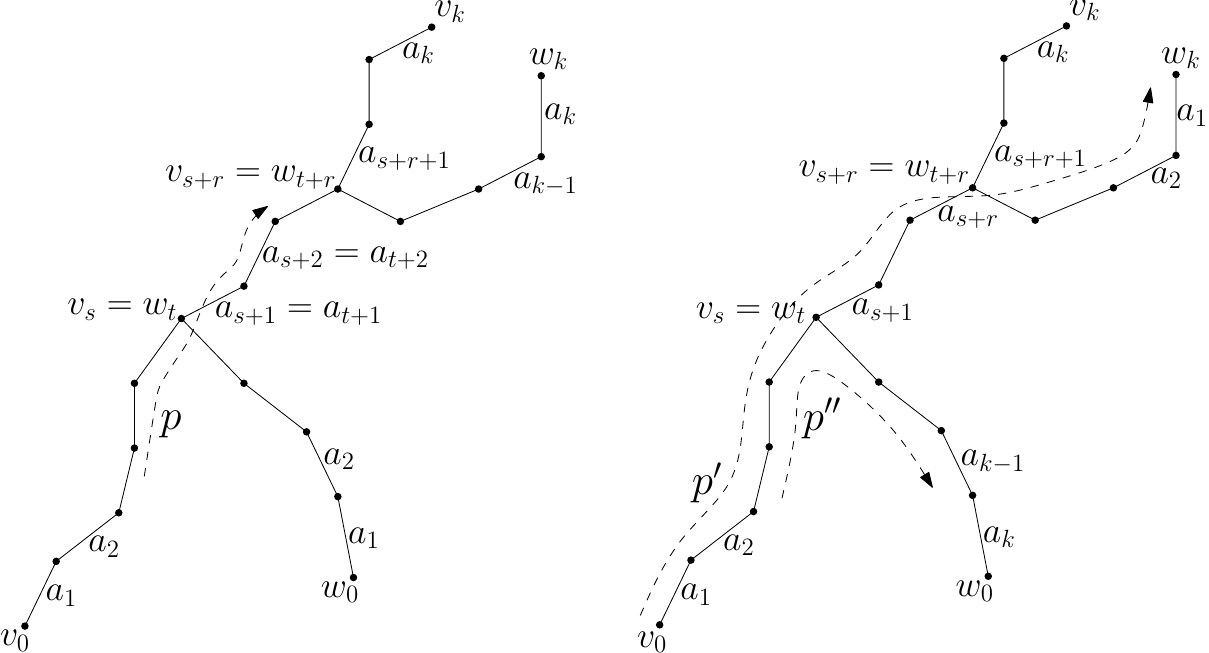}
\caption{Case IIa (left) and Case IIb (right)}
\label{f:IIb}
\end{figure}
In this case we consider separately two subcases:
\medskip

(i) If $s=t$ then we consider the path $(v_{s-1},v_s,w_{s-1})$, if $s>0$, or the path $(v_{s+r+1},v_{s+r},w_{s+r+1})$ otherwise.
We obtain the labelling:
\begin{align*}
(\underline{a_{s}},\underline{\bar a_s})\; \mr{or}\; (\underline{\bar a_{s+r+1}},\underline{a_{s+r+1}}).
\end{align*}
\medskip

(ii) If $s> t$ then we obtain a path $p:=(v_t,v_{t+1},\ldots, w_t,w_{t+1},\ldots,w_s)$ of length bounded from above by
\begin{align}
\label{e:m10}
2k\leqslant 2\lambda \, \gi \tn.
\end{align}
Its labelling has the form:
\begin{align}
\label{e:m9}
(\underline{a_{t+1},a_{t+2},\ldots,a_s},\underline{a_{t+1},a_{t+2},\ldots,a_s}).
\end{align}
(The two above cases are `repetitive' labellings as in \cite{AloGry}.)

\medskip

\noindent
\emph{(Case IIb): The labelling of a directed edge $w_{i+1}w_{i}$ is $a_{k-i}$} (see Figure~\ref{f:IIb} on the right).
In this case we consider separately three subcases:
\medskip

(i) If $s>k/3$ and $t< 2k/3$ then we consider the path $p':=(v_0,\ldots, v_s,w_{t+1},\\ \ldots, w_k)$. Its length is bounded from above by 
(\ref{e:m10}), and its labelling has the form:
\begin{align*}
(\underline{a_{1},a_{2},\ldots,a_{q}},\ldots,\underline{\bar a_{q},\ldots,\bar a_{2},\bar a_1}),
\end{align*}
with
\begin{align}
\label{e:m12}
q > k/3. 
\end{align}
\medskip

(ii) If $t\geqslant 2k/3$, then $s\geqslant t \geqslant 2k/3$. In this case we consider the path $p'':= (v_{k-t},\ldots,v_s,w_{t-1},\ldots,w_{k-s})$. Its length is bounded from above by (\ref{e:m10}), and its labelling has the form:
\begin{align}
\label{e:m13}
(\underline{a_{k-t+1},\ldots,a_{s}},\underline{a_{k-t+1},\ldots,a_s}).
\end{align}
\medskip

(iii) If $s\leqslant k/3$ then for $s+r<2k/3$ we are in one of the previous cases (with $s>k/3$) after changing indexes $i$ to $k-i$ and renaming.
Thus we may assume that $s+r\geqslant 2k/3$. 
Then we consider the path $p':=(v_0,v_1,\ldots, v_s,w_{t+1},\ldots, w_k)$. Its length is bounded from above by (\ref{e:m10}), and its labelling has the form:
\begin{align}
\label{e:m14}
(\underline{a_{1},a_{2},\ldots,a_{q}},\ldots,\underline{\bar a_{q},\ldots,\bar a_{2},\bar a_1}),
\end{align}
with
\begin{align}
\label{e:m15}
q \geqslant k/3. 
\end{align}

\begin{lemma}
\label{l:mLLL}
Let $E:=\lambda/(16\lambda+8A)$ and $F:=(2\lambda+A)\gi \tn$. Assume that there are two different paths in $(\tn,m_n)$, of length at least $\lambda \, \gi \tn$, with the same labelling. Then one of the following situations happens:
\begin{enumerate}
\item[(A)]
there is a path $p$ with the labelling $(\underline{a_{i_1},a_{i_2},\ldots,a_{i_q}},\ldots,\underline{a_{i_1},a_{i_2}\ldots,a_{i_q}})$, for
$|p|\leqslant F$ and $q\geqslant E|p|$;
\item[(B)]
there is a path $p$ with the labelling $(\underline{a_{i_1},a_{i_2},\ldots,a_{i_q}},\ldots,\underline{\bar a_{i_q},\ldots,\bar a_{i_2},\bar a_{i_1}})$, for 
$|p|\leqslant F$ and $q\geqslant E|p|$.
\item[(C)]
there is a path $p$ with the labelling $(\underline{a_{i}},\underline{\bar a_{i}})$.

\end{enumerate}
\end{lemma}
\begin{proof}
We show that all the cases analyzed earlier in this section lead to (A), (B) or (C). This covers all the possible configurations.
\medskip

(A) corresponds to the cases: Ib(ii), IIa(ii) and IIb(ii). The estimates on $|p|$ and $q$ follow then from: formula (\ref{e:m1}) and formula (\ref{e:m7}), or from (\ref{e:m9}), or from (\ref{e:m13}), and from the fact that 
\begin{align}
\label{e:m16}
E=\frac{\lambda}{16\lambda+8A}< \frac{1}{16}.
\end{align}
\medskip

(B) corresponds to one of the cases:  Ia, Ib(i), IIb(i) or IIb(iii). The estimates on $|p|$ and $q$ follow then from: (\ref{e:m1}) and (\ref{e:m3}) or (\ref{e:m5}), or from (\ref{e:m10}) and (\ref{e:m12}) or (\ref{e:m15}), using (\ref{e:m16}).
\medskip

(C) corresponds to Case IIa(i).
\end{proof}

Now we show, similarly as in the preceding Subsection~\ref{s:LLL}, that there exists a labelling $(\tn,l'_n)$ such that none of the patterns
(A), (B) or (C) from Lemma~\ref{l:mLLL} appears. This will imply that no two different paths in $\tn$ of length at least $\lambda \, \gi \tn$ have the same $l'_n$--labelling. This will also mean that $l'_n$ is reduced. The labelling $l'_n$ will use $L'$ labels.
Define $L'$ as (here $e$ denotes the Euler constant)
\begin{align}
\label{e:m20}
L':=\left \lceil (4De^4)^{\frac{1}{E}} \right \rceil,
\end{align}
where $E=\lambda/(16\lambda +8A)$ is the constant from Lemma~\ref{l:mLLL}.
Call a labelling of a path $p$ \emph{bad} if it is of the form (A), (B) or (C) as in Lemma~\ref{l:mLLL}. 
Let $M_i$ denote the number of possibilities of labelling a fixed simple path of length $i$
in a bad way by $L'$ letters.
Let $N_i$ denote the number of possibilities of labelling a fixed simple path of length $i$ by $L'$ letters.
Observe that
\begin{align}
\label{e:m21}
M_i\leqslant 2L'^{(1-E)i},
\end{align}
and
\begin{align}
\label{e:m22}
N_i= L'^{i}.
\end{align}
\begin{lemma}
\label{l:scmLLL}
There exists a labelling $(\tn,l'_n)$ with $L'$ labels such that, for $2\leqslant i\leqslant F=(2\lambda +A)\gi \tn$ no $l'_n$--labelling of a path of length
$i$ is bad.
\end{lemma}
\begin{proof}
We use the Lov\'asz Local Lemma \ref{l:LLL} as in the proof of Lemma~\ref{l:scLLL}. 
Randomly label the edges of $\tn$ with $L'$ labels. 
For a path $p$ in $\tn$ of length $i$, let $A(p)$ denote the event that its labelling is bad.
Set $\cA_i=\{ A(p) :  p \; \mr{is\; a\; path\; of\; length}\; i \;\mr{in} \; \tn \}$. Recall (see Lemma \ref{l:LLL}) that $p_i$ denotes the probability $\mr{Pr}(A)$ for every $A\in \cA_i$. Then, by (\ref{e:m21}) and  (\ref{e:m22}), we have
\begin{align}
\label{e:m23}
p_i\leqslant \frac{2L'^{(1-E)i}}{L'^{i}}\leqslant \left( \frac{2}{L'^E}\right)^i.
\end{align}
 Each path of length $i$ shares an edge with not more than $ijD^{j}$ paths of length $j$, so that we may take 
$\Delta_{ij}=ijD^{j}$.
Let $a_i=a^{-i}$, where $a=2D$. 
Then, by using subsequently: formulas (\ref{e:m23}) and (\ref{e:m20}), the definition of $a_i$, the fact that $\sum_{j=1}^{\infty} j/2^j =2$, the definitions of $a$, $\Delta_{ij}$, and $a_j$, we obtain:

\begin{align*}
\begin{split}
p_i&\leqslant \left( \frac{2}{L'^E}\right)^i\leqslant 2^{-i}D^{-i}e^{-4i}
= a_i \exp{\left(-2\sum_{j=1}^{\infty}i\frac{j}{2^j} \right)}
\\ &<a_i \exp{\left(-2\sum_{j}ij\left(\frac{D}{a}\right)^{j} \right)}
 =a_i \exp{\left(-2\sum_{j} \Delta_{ij}a_j \right)}
\\ &= a_i\prod_j e^{-2a_j\Delta_{ij}}
\end{split}
\end{align*}
Since, by $a_j\leqslant 1/2$, we have $e^{-2a_j}\leqslant (1-a_j)$ (see the end of the proof of Lemma~\ref{l:scLLL}), we obtain finally
\begin{align*}
p_i\leqslant a_i\prod_j (1-a_j)^{\Delta_{ij}}.
\end{align*}
Therefore the hypotheses of the Lov\'asz Local Lemma are fulfilled, and we conclude that there exists a labelling $l'_n$ as required.
\end{proof}

\begin{lemma}[$C'(\lambda)$--small cancellation labelling of $\tn$]
\label{l:scm}
The labelling $(\tn,l'_n)$ with $L'$ labels is reduced and no two paths in $\tn$ of length at least $\lambda \, \gi \tn$ have the same $l'_n$--labelling.
\end{lemma}
\begin{proof}
The labelling $(\tn,l'_n)$ is reduced because the situation (C) from Lem\-ma \ref{l:mLLL} does not appear. The second assertion follows from Lemma~\ref{l:mLLL} and
the fact that  none of the situations (A) and (B) appears for $l'_n$, by Lemma~\ref{l:scmLLL}.
\end{proof}


\subsection{Small cancellation labelling of \texorpdfstring{$\Theta$}{Theta}}
\label{s:fscl}
Let $(\Theta,l)=((\tn,l_n))_{n\in \bN}$ and $(\Theta,l')=((\tn,l'_n))_{n\in \bN}$ be the labellings with, respectively,
$L$ and $L'$ labels
given by Lemma~\ref{l:sc1} and Lemma~\ref{l:scm}. 
Let $(\Theta,m)=((\tn,m_n))_{n\in \bN}$ be a labelling being the product of $(\Theta,l)$ and $(\Theta,l')$.
That is, to every directed edge
$e$ in $\tn$ we assign a pair $(l(e),l'(e))$.
By Lemma~\ref{l:sc1} and Lemma~\ref{l:scm} we obtain the following main technical result of the paper (see Theorem 1 in Introduction).

\begin{theorem}[{$C'(\lambda)$--small cancellation labelling of $\Theta$}]
\label{l:c2LLL}
The labelling $(\txT,m)$ is reduced and no
$m_n$--labelling of a path of length at least $\lambda\, \gi \txT_n$ in  $\txT_n$ appears as the $m$--labelling of some other path in $\txT$.
\end{theorem}

\begin{remark}
	\label{r:recurs}
	Observe that, for every $n$, a given finite labelling $(\txT_1,l_1), 
	(\txT_2,l_2),\\ \ldots, (\txT_{n},l_n)$
	from Subsection~\ref{s:LLL} can be extended to $(\txT_{n+1},l_{n+1})$ by using the brute force
	algorithm (see Lemma~\ref{l:scLLL}). The same holds for the labelling $(\txT,l')$ from Subsection~\ref{s:mLLL}. Therefore, if the sequence $\txT$ of finite graphs is recursive,  
	the small cancellation labelling $(\txT,m)$, as well as the resulting small cancellation presentation (see Section~\ref{s:grexp}) are recursive. This observation is important in particular in view of applications described in Subsection~\ref{s:exman} below.
	
	Recursive sequences of finite graphs $\txT$ (satisfying our assumptions from the beginning of Section~\ref{s:constr}) exist. Examples are expander graphs given by Cayley graphs of some finite 
	linear groups; see Subsection~\ref{s:discus} for some details.
\end{remark}


\subsection{Remarks on the Gromov labelling}
\label{s:discus}
In this subsection we recall the idea of Gromov's construction of a `small cancellation' labelling of some expanders \cite{Gro}, following its exposition presented in \cite{AD}. 
We remark that a construction of a labelling as in Lemma~\ref{l:scm} could be obtained using Gromov's construction. (Let us also remark that it could be concluded from \cite[Proposition 7.4]{OW}.)
Further, we explain why one cannot obtain the small cancellation
labelling out of the one of Gromov, that is, why Lemma~\ref{l:sc1} does not hold for the generic labelling.
\smallskip

For  primes $p\neq q$ congruent to $1$ modulo $4$ and with the Legendre symbol $\left( \frac{p}{q} \right)=-1$, let $X^{p,q}$ be the Cayley graph of the projective linear group $PGL_2(q)$, for some particular set of $(p+1)$ generators, as in \cite[Section 7.2]{AD}. 
Fix $p$.
Throughout this subsection we consider subsequences of the sequence $\Theta=(\Theta_n)_{n\in \bN}$, where $\Theta_n=X^{p,q_n}$,
with $q_n$ denoting the $n$--th prime. 
Then the family $\Theta$ is an expander with the constant degree $D:=p+1$, with $\gi \Theta_n \to \infty$, as $n\to \infty$, and for which 
there exists a constant $A$ such that (\ref{e:dig}) holds.
Gromov \cite{Gro} constructs a labelling $(\Theta,l')=((\tn,l'_n))_{n\in \bN}$ (also for a class of expanders) satisfying some small cancellation conditions.

Let $G_0$ be the free group generated by a finite set $S$. The labelling $l'$ is a map $l'\colon \Theta \to W$ onto the bouquet of $|S|$
oriented loops labelled by $S$ (whose fundamental group is $G_0$). 
The labelling $(\Theta,l')$ is obtained inductively. 

We begin with $\Theta_1$ and we find a labelling $l'_1\colon \Theta_1 \to W$ satisfying
some small cancellation conditions. We obtain a hyperbolic group $G_1$ being the quotient of the free group $G_0$ by the normal subgroup
generated by images of $l'_1$. At the inductive step, having a hyperbolic group $G_{n-1}$ generated by $S$, a random (generic) labelling
$l'_{n}\colon \Theta_n \to W$ satisfies the \emph{very small cancellation conditions} for the small cancellation constant arbitrarily close to $0$ by \cite[Proposition 5.9]{AD}.
The labelling $l'$ may be used to construct a labelling with properties as in Lemma~\ref{l:scm}.
\smallskip

However, out of $l'$ one cannot derive the required small cancellation labelling as in Theorem~\ref{l:c2LLL}. Since at each step Gromov's labelling is the generic labelling appearing as $\gi \tn \to \infty$, it is clear that the following holds: For any fixed labelling of a path of a fixed length, with overwhelming probability this labelling will appear among labellings of $\tn$ as $n\to \infty$. In particular, labellings of all cycles in graphs obtained at earlier inductive steps will appear as labellings of paths in later steps. 
This is the reason why Gromov's labelling is not a graphical small cancellation labelling.
Let $G$ be the group being the limit of $(G_n)_{n\in \bN}$.
Observe that endpoints of a simple path in $\tn$ (for large $n$) labelled the same as a cycle in some $\Theta_m$ with $m\ll n$ will be mapped to a same point by the map $\tn \to G$ defined by the labelling.
Since there are such paths of arbitrarily large length, this labelling does not define a coarse embedding of $\Theta$ into $G$. There is only a weak embedding or, stronger, a map $f\colon \Theta \to G$ satisfying the following condition: for $x,y\in \tn$ one has $d_G(f(x),f(y))\geqslant Bd_{\Theta}(x,y)-c_n$, where $B$ is a universal constant, and additive constants $c_n>0$ grow to infinity with $n\to \infty$; see \cite[Section 4.8]{Gro} and \cite[Theorem 7.7]{AD}.


\section{Groups with \texorpdfstring{$\Theta$}{Theta} in Cayley graphs}
\label{s:grexp}

In this section we construct groups, such that $\Theta$ embeds isometrically into their Cayley graphs -- this means that the vertex
set of every connected component $\tn$ embeds isometrically.
The groups are defined by \emph{graphical small cancellation presentations}. 
The graphical small cancellation theory is a straightforward generalization of the classical small cancellation theory -- see e.g.\ the book by Lyndon and Schupp \cite{LSch} for the exposition of the latter.  
The introduction of the graphical theory is attributed to Gromov \cite{Gro}, but the methods had appeared implicitly before e.g.\ in the work of Rips and Segev \cite{RipsSegev}. 
In order to apply small cancellation we use the sequence $\stxT$ as follows. 
Let $\Gamma$ be a finite graph and let $(\varphi_n \colon \stxT_n \to \Gamma)_{n\in \mathbb N}$ be a family of local isometries of graphs. They form a
\emph{graphical presentation}
\begin{align}\label{eq:gpres} 
\langle \Gamma \; | \; \stxT \rangle,
\end{align}
defining a group $G:=\pi_1(\Gamma)/\langle\langle \varphi_{\ast}(\pi_1(\stxT_n))_{n\in \mathbb N} \rangle\rangle$.
In our case we choose $\Gamma$ to be a bouquet of loops
 with local isometries $\varphi_n$ corresponding to the labellings $\stxm_n$. Each loop in the bouquet 
 corresponds to one generator of $G$.

\subsection{\texorpdfstring{$C'(\lambda)$}{C'(lambda)}--small cancellation complexes.}
\label{s:c'cpl}
This subsection follows closely \cite[Section 2]{AO}.
Here we describe the spaces that we will work with further. 
Let $(\varphi_i \colon r_i \to \xj)_{i\in \mathbb N}$ be a family of local isometries of finite graphs $r_i$.
We will call these finite graphs \emph{relators}.
The \emph{cone} over the relator $r_i$ is the quotient space $\mr{cone}\, r_i:=(r_i \times [0,1])/\{ (x,1) \sim (y,1)\}$.
The main object of our study in this section is the \emph{coned-off space}:
\begin{align*}
X:=X^{(1)}\cup_{(\varphi_i)} \bigcup_{i\in \mathbb N} \mr{cone}\,r_i,  
\end{align*}
where $\varphi_i$ is the map $r_i\times \{ 0 \} \to \xj$.
We assume that $X$ is simply connected.
The space $X$ has a natural structure of a CW complex and we call $X$ a `complex'. 
If not specified otherwise, we consider the \emph{path metric}, denoted by $d(\cdot,\cdot)$, defined on the $0$--skeleton $X^{(0)}$ of $X$ by (combinatorial) paths in the $1$--skeleton $X^{(1)}$. \emph{Geodesics} are the shortest paths in $X^{(1)}$ for this metric. (In other words, a geodesic between vertices $p,q\in X^{(0)}$ is a shortest sequence $p_0:=p,p_1,\ldots,p_k:=q$ of vertices such that $p_i$ and $p_{i+1}$ are 
connected by an edge in $X^{(1)}$.)

\medskip
A \emph{path in $X$} is a locally injective simplicial map $p\to X$ from a graph $p$ homeomorphic to a segment. 
A path $p \to X$ is a \emph{piece} if there are relators $r_i,r_j$ such that $p\to X$ factors as $p \to r_i \stackrel{\varphi_i}{\longrightarrow} X$ and as $p\to r_j \stackrel{\varphi_j}{\longrightarrow}  X$, but there is no isomorphism $r_i \to r_j$ that makes the following diagram commutative.
$$  \begindc{\commdiag}[20]
\obj(12,1)[a]{$r_i$}
\obj(35,2)[b]{$X$}
\obj(35,1)[b']{}
\obj(35,16)[c]{$r_j$}
\obj(35,17)[c']{}
\obj(12,16)[d]{$p$}
\obj(12,17)[d']{}
\mor{a}{b'}{}
\mor{c}{b}{}
\mor{a}{c}{}
\mor{d'}{c'}{}
\mor{d'}{a}{}
\enddc  $$

\noindent
 This means that $p$ occurs in $r_i$ and $r_j$ in two essentially distinct ways.

For $\lambda \in (0,1)$, we say that the complex $X$ satisfies the \emph{$C'({\lambda})$--small cancellation condition} (or that $X$ is a \emph{$C'({\lambda})$--complex}) if 
every piece $p\to X$ factorizing through $p\to r_i \stackrel{\varphi_i}{\longrightarrow} X$ has length (that is, the number of edges in $p$) strictly less than $\lambda\,  \gi r_i$.

For a given graphical presentation $\langle \Gamma \; | \; \stxT \rangle$, we define an associated complex $X$ as follows. 
The \emph{coned-off space} is obtained by gluing, using the (labelling) maps $\stxT_n \to \Gamma$, cones over graphs $\stxT_n$ to $\Gamma$. 
The fundamental group of this space is $G$. 
The Cayley graph of $\langle \Gamma \; | \; \stxT \rangle$ is the $1$--skeleton $X'^{(1)}$ of the universal cover $X'$ of the coned-off space. We define maps $\varphi_i' \colon r_i \to \xj$ as lifts of the maps $\stxT_n \to \Gamma$. In particular, the graphs $r_i$ are copies of graphs $\stxT_n$, for various $n$. 
Finally, we define $X$ as the quotient complex of $X'$ where we identify the cones attached
by $\varphi_i'\colon r_i \to X'^{(1)}$ and $\varphi_j'\colon r_j \to X'^{(1)}$ when there is an isomorphism
of labelled (by the labelling induced by  $\stxT_n \to \Gamma$) graphs $r_i\to r_j$ such that $\varphi_i'$ factors as $r_i\to r_j \stackrel{\varphi_j'}{\longrightarrow} X'^{(1)}$. The maps $\varphi_i \colon r_i \to X$ are the compositions of the maps $\varphi_i'$ with the quotient map $X'\to X$.
If the complex $X$ is a $C'({\lambda})$--complex then we call the presentation $\langle \Gamma \; | \; \stxT \rangle$ a \emph{graphical $C'({\lambda})$--small cancellation presentation}.

The following lemma attributed to Gromov is crucial for our results.
\begin{lemma}[{\cite[Theorem 1]{Oll}} and {\cite[5.10]{Gruber}}]
\label{l:remb}
For the Cayley graph $\xj$  of a {graphical $C'(1/6)$--small cancellation presentation} $\langle \Gamma \; | \; \stxT \rangle$ the maps
$\varphi_i\colon r_i \to \xj$ are isometric embeddings.
\end{lemma}


\subsection{The groups}
\label{s:grex}
In this section we use the labelling $(\Theta,m)$ as in Theorem~\ref{l:c2LLL}, obtained for $\lambda \leqslant 1/24$.

\begin{theorem}[Groups containing $\Theta$]
\label{t:mainex}
Let $G$ be the group defined by the graphical presentation $\langle \Gamma \, | \, \stxT \rangle$, where the local isometries
$\stxT_n \to \Gamma$ are defined by labellings $\stxm_n$. Then each $\stxT_n$ embeds isometrically into the Cayley graph of $G$ 
given by $\langle \Gamma \, | \, \stxT \rangle$.
\end{theorem}
\begin{proof}
By Lemma~\ref{l:c2LLL}, the complex $X$ associated with $\langle \Gamma \, | \, \stxT \rangle$ satisfies the $C'(\lambda)$--small cancellation condition. By Lemma~\ref{l:remb}, every $r_i$ embeds isometrically into $X^{(1)}$. 
\end{proof}
In the following subsections we study more specific applications of Theorem~\ref{t:mainex}.
\subsubsection{Groups containing expanders}
\label{s:grex}
Expanders do not admit coarse embeddings into Hilbert spaces \cite{Mat}. It follows from \cite[Section 7]{HiLaSka} that groups containing coarsely expanders do not satisfy the Baum-Connes conjecture with coefficients. The following is a direct consequence of the results above and Theorem~\ref{t:mainex}.
\begin{corollary}
\label{c:Hilbert}
If $\Theta$ is an expanding sequence of graphs then the group $\langle \Gamma \, | \, \stxT \rangle$ is not coarsely embeddable into a Hilbert space, and it does not satisfy the Baum-Connes conjecture with coefficients. 
\end{corollary}
{The next result has been proved in \cite{WillettYu1} for groups with coarsely embedded expanders. As explained in Subsection~\ref{s:discus}, for Gromov's monster only the weak embedding is established. Therefore, our construction provides
the first examples of groups, for which the conclusion of the following corollary holds.
\begin{corollary}[{\cite[Corollary 1.7]{WillettYu1}}]
\label{c:WiYu}
Let $G$ be a group defined by the gra\-phical presentation $\langle \Gamma \, | \, \stxT \rangle$, where the local isometries
$\stxT_n \to \Gamma$ are defined by labellings $\stxm_n$, and where $\stxT$ is the sequence of expanding graphs with growing girth.
Let $X$ be the image of the isometric embedding of $\stxT$ into the Cayley graph  $Y$ of $G$. For each $n\in \bN$, let $X_n =
\{y\in Y\,  |\, d_Y(y,X) \leqslant  n\}$. Let $A_n = l^{\infty}(X_n,\mathcal K)$ and $A = \lim_{n\to \infty}l^{\infty}(X_n,\mathcal K)$, where $\mathcal K$ is the algebra of compact operators on a given infinite dimensional separable Hilbert space. Then the right action of $G$
on $Y$ gives $A$ the structure of a $G$--$C^{\ast}$--algebra and:
\begin{enumerate}
\item
the Baum-Connes assembly map for $G$ with coefficients in $A$ is injective;
\item
the Baum-Connes assembly map for $G$ with coefficients in $A$ is not surjective;
\item
the maximal Baum-Connes assembly map for $G$ with coefficients in $A$ is an isomorphism.
\end{enumerate}
\end{corollary}
Similarly, the existence of groups with coarsely embedded expanders is crucial for \cite[Section 7]{BGW}.}
\subsubsection{Exotic aspherical manifolds}
\label{s:exman}
Sapir \cite{Sapir} developed a technique of embedding groups with combinatorially aspherical recursive
presentation complexes into groups with finite combinatorially aspherical
presentation complexes. 
The presentation (\ref{eq:gpres}) defined by the labelling $(\Theta,m)$ from Theorem~\ref{l:c2LLL} is aspherical; see e.g.\ \cite{Oll}.
It is also recursive -- the brute force algorithm can be used to find the labelling $(\Theta, m)$ -- see Remark~\ref{r:recurs}.
By embedding the group $\langle \Gamma \, | \, \stxT \rangle$ from Corollary~\ref{c:Hilbert} into a finitely presented group we obtain the first examples of such groups coarsely containing expanders.
Therefore, using Sapir's techniques and Theorem~\ref{t:mainex} we obtain the first examples of manifolds as follows.
\begin{corollary}
\label{c:Sapir} 
There exist closed aspherical manifolds of dimension $4$ and higher
whose fundamental groups contain coarsely embedded expanders.
\end{corollary}
%

\section{Walls}
\label{s:wall}
In this section and in the next Section~\ref{s:wlac} we develop a theory that will allow us in Section~\ref{s:pwna} to show that the
group we construct there acts properly on a space with walls. We use here the notation from Section~\ref{s:c'cpl} concerning $C'(\lambda)$--complexes. The current section is very similar to \cite[Section 3]{AO}.
\medskip

Recall, that for a set $Y$ and a family $\mathcal W$ of partitions (called \emph{walls}) of $Y$ into two parts, the pair $(Y,\mathcal W)$ is called a \emph{space with walls} \cite{HP} if the following holds. For every two distinct points $x,y\in Y$ the number of walls separating $x$ from $y$ (called the \emph{wall pseudo-metric}), denoted by $\dw(x,y)$, is finite.

In this section, following the method of Wise \cite{W-qch} (see also \cite{W-ln}), we equip the $0$--skeleton of a $C'(\lambda)$--complex with the structure of space with walls. To be able to do it we have to make some assumptions on relators.

A \emph{wall} in a graph $\Gamma$ is a collection $w$ of edges such that removing all open edges of $w$ decomposes $\Gamma$ in
exactly two connected components. We call $\Gamma$ a \emph{graph with walls}, if every edge belongs to a unique wall. This is a temporary abuse of notations with respect to `walls' defined as above, which will be justified later. 

If not stated otherwise, we assume that for a $C'(1/24)$--complex $X$ associated to a graphical presentation as explained in Subsection~\ref{s:c'cpl}, with given relators $r_i$, each graph $r_i$ is  a graph with walls. 
In the current section and in the following Section~\ref{s:wlac}, using Lemma~\ref{l:remb}, we treat the relators $r_i$ as isometric subgraphs of $X$. This slight abuse of notation should not lead to confusion.
Following \cite[Section 5]{W-qch}, we define walls in $\xj$ as follows: Two edges are in the same wall if they are in the same wall in some relator $r_i$. This relation is then 
extended to an equivalence relation on the set of all edges of $X$. In particular, every edge is contained in a wall (possibly consisting of only that edge).

In general, the above definition may not result in walls for $X^{(0)}$. We require some further assumptions on
walls in relators, which are formulated below. 

\begin{definition}[($\beta,\Phi$)--separation]
\label{d:sep}
For $\beta \in (0,1/2]$ and a homeomorphism $\Phi \colon [0,+\infty)\to [0,+\infty)$, a graph $r$ with walls satisfies the \emph{($\beta,\Phi$)--separation property}
if the following two conditions hold:
\medskip

\noindent
\emph{\underline{$\beta$--condition}}: for every two edges $e,e'$ in $r$ belonging to the same wall we have
\begin{align*}
d(e,e')+1\geqslant \beta \, \gi\,r.
\end{align*}

\noindent
\emph{\underline{$\Phi$--condition}}: for every geodesic $\gamma$ in $r$, the number of edges in $\gamma$ whose walls 
have only one edge in common with $\gamma$ (and thus, in particular, separate the end-points of $\gamma$) is at least $\Phi(|\gamma|)$.
\medskip

A complex $X$ satisfies the \emph{\bds property} if every its relator does so.
\end{definition}

\begin{prop}[{\cite[Lemma 3.3]{AO}}]
\label{l:walls}
For every $\beta \in (0,1/2]$ there exists $\lambda\leqslant 1/24$, such that for every $C'(\lambda)$--complex $X$ satisfying the $\beta$--condition 
the following holds. Removing all open edges from a given wall decomposes $\xj$ into exactly two connected components.
The family of the corresponding partitions induced on $X^{(0)}$ defines the structure of a space with walls $(\xz,\mathcal W)$.
\end{prop}

In what follows we assume that a $C'(\lambda)$--complex $X$ is as in the proposition.
We recall further results on $(\xz,\mathcal W)$ that will be extensively used in Section~\ref{s:wlac}.

For a wall $w$, its \emph{hypergraph} $\Gamma_w$ is a graph defined as follows (see \cite[Definition 5.18]{W-qch} and \cite{W-sc}). 
There are two types of vertices in $\Gamma_w$ (see e.g.\ Figure~\ref{f:C}):
\begin{itemize}
\item
\emph{edge-vertices} correspond to edges in $w$,
\item
\emph{relator-vertices} correspond to relators containing edges in $w$.  
\end{itemize}
An \emph{edge} in $\Gamma_w$ connects an edge-vertex to a relator-vertex whenever the corresponding
relator contains the given edge.

The \emph{hypercarrier} of a wall $w$ is the $1$--skeleton of the subcomplex of $X$ consisting of all relators containing edges in $w$ or of a single edge $e$ if $w=\{ e \}$.
The following theorem recalls the most important facts concerning walls; see \cite[Subection 3.3]{AO}.
\begin{theorem}  \label{l:tree}
Each hypergraph is a tree.
Relators and hypercarriers are convex subcomplexes of $\xj$.
\end{theorem}

Observe that if edges $e,e'$ are in the same relator $r$ and, moreover, they belong to the same wall in $\xj$ then
$e,e'$ belong to the same wall in $r$ (for the initial wallspace structure on $r$). 

\section{Proper lacunary walling}
\label{s:wlac}
In this section we introduce the condition of proper lacunary walling (see Definition~\ref{d:cond}), and we show that for complexes satisfying this condition
the wall pseudo-metric is proper; see Theorem 3 in Introduction and Theorem~\ref{p:lsp} below.
We follow the notation from Section~\ref{s:c'cpl} and Section~\ref{s:wall}. The section is based on  \cite[Section 4]{AO}. Note however that whereas the proper lacunary walling condition from the current paper is weaker than the corresponding lacunary walling condition from \cite{AO}, consequences of the former are also weaker: We obtain properness of the wall pseudo-metric, and in \cite{AO} a linear separation property is established. Unfortunately, we are not able to use the lacunary walling condition from \cite{AO}
to construct corresponding groups (and we believe it may be not possible). Therefore, for the sake of the constructions in this article we introduced the proper lacunary walling conditions studied further in this section. Note also that the notions used here may be sometimes quite different from the ones used in \cite{AO},
hence we have to provide new proofs of corresponding results.

\medskip

For a relator $r$ and a vertex $v\in r$, let $P_v(r)$ denote the number of edges in $\bigcup_{r'}r\cap r'$, where
$r'$ varies through all relators $r'\neq r$ containing $v$. Let $P(r)$ denote the maximal number among $P_v(r)$
for $v\in r$.

\begin{definition}[{Proper lacunary walling}]
\label{d:cond}
Let $\beta \in (0,1/2]$, and let $D$ be a natural number larger than $1$. 
Let  $0<\lambda< \beta/2$ be as in Proposition~\ref{l:walls} (that is, such that $(X^{(0)},\mathcal W)$ is a space with walls). 
Let $\Phi,\Omega,  \Delta \colon [0,+\infty) \to [0,+\infty)$ be homeomorphisms.
We say that $X$ satisfies the \emph{proper lacunary walling condition} if:
 
\begin{itemize}
\item
$X^{(1)}$ has degree bounded by $D$;
\item
(Small cancellation) $X$ satisfies the $C'({\lambda})$--condition;
\item
(Separation) $X$ satisfies the \bds property;
\item
(Lacunarity) $\Phi((\beta - \lambda)\, \gi  r_i)-4\, P(r_i) \geqslant \Omega(\gi r_i) $;
\item
(Large girth) $\gi r_i \geqslant \Delta(\di r_i)$. 
\end{itemize}  
\end{definition}

\begin{center}
	{\bf For the rest of this section we assume that the complex $X$ satisfies the proper lacunary walling condition from Definition~\ref{d:cond} with parameters $\beta,D,\lambda,\Phi,\Omega,  \Delta$.}
\end{center}

It is clear that $\dw(p,q)\leqslant d(p,q)$. The rest of this section is devoted to bounding the wall pseudo-metric $\dw$ from below. 
Let $\gamma$ be a geodesic in $X$ (that is, in its $1$--skeleton $\xj$) with endpoints $p,q$. Let $A(\gamma)$ denote the set of edges in $\gamma$ whose walls meet $\gamma$ in only one edge (in particular such walls separate $p$ from $q$). Clearly $\dw(p,q)\geqslant |\ag|$.
We thus estimate $\dw(p,q)$ by closely studying the set $\ag$. The estimate is first provided locally (in Subsection~\ref{s:local} below) and then we use the local bounds to obtain a global one. In what follows, by $E(Y)$ we denote the set of edges of a subcomplex $Y\subseteq X$.
\medskip

We begin with an auxiliary lemma.
Let $r$ be a relator. Since, by Theorem~\ref{l:tree}, $r$ is convex in $X$, its intersection with $\gamma$ is an interval $p'q'$, with $p'$ lying closer
to $p$; see Figure~\ref{f:C}. 
\begin{figure}[h!]
\centering
\includegraphics[width=0.9\textwidth]{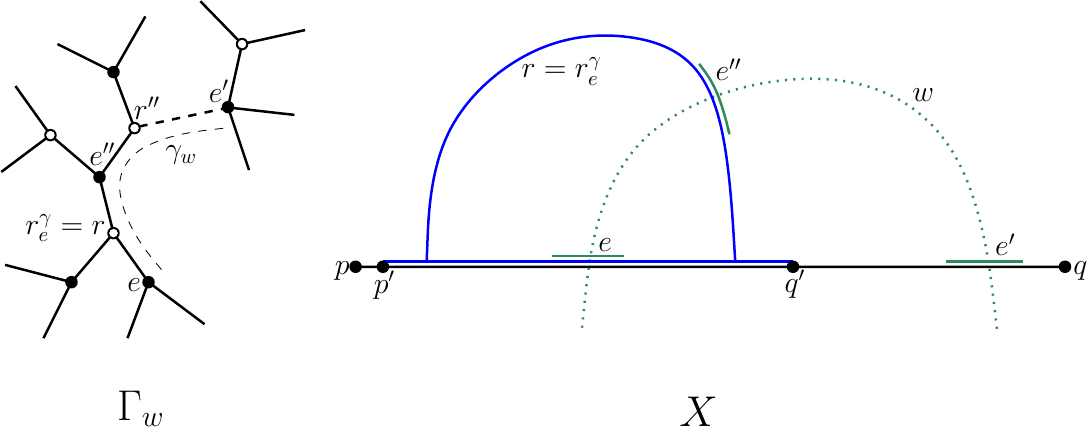}
\caption{Lemma~\ref{l:C}.}
\label{f:C}
\end{figure}
Consider the set $C$ of edges $e$ in $p'q'$, whose walls $w$ meet $\gamma$ at least twice and, moreover, have the following properties. Let $e'\in w$ (considered as an edge-vertex in the hypergraph $\Gamma_w$ of the wall $w$) be a closest vertex to $e$ in $\Gamma_w$, among edges of $w$ lying on $\gamma$. In the hypergraph $\Gamma_w$ of the wall $w$, which is a tree by Theorem~\ref{l:tree}, consider the unique geodesic $\gamma_w$ between vertices $e$ and $e'$. We assume that there are at least two distinct relator-vertices on $\gamma_w$, one of them being~$r$.

\begin{lemma}
\label{l:C}
In the situation as above we have $|C|\leqslant 2 P(r)$.
\end{lemma}
\begin{proof}
The proof is basically the same as the one of \cite[Lemma~4.2]{AO}. Since we need to express the statement in a slightly different way
we recall the proof for completeness.
Suppose that $q'$ lies between $e$ and $e'$ (on $\gamma$). 
Let $e''\neq e$ be the edge-vertex on $\gamma_w$ adjacent to $r$ and, subsequently, let $r''\neq r$ be the relator-vertex on $\gamma_w$ 
adjacent to $e''$ --- see Figure~\ref{f:C}.
By convexity and the tree-like structure of the hypercarrier of $w$ containing $e$ and $e'$ (see Theorem~\ref{l:tree}) we have that $q'\in r''$. Since $r\cap r''$ contains both
$e''$ and $q'$, we have that the number of edges $e''$ as above is at most $P(r)$.
The same number bounds the quantity of the corresponding walls.
By our assumptions, every such wall contains only one edge in $p'q'$.
Thus, the number of edges $e$ as above is at most $P(r)$.
Taking into account the situation when $p'$ lies between $e$ and $e'$ we have $|C|\leqslant 2 P(r)$.
\end{proof}


\subsection{Local estimate on \texorpdfstring{$|\ag|$}{|A(gamma)|}.}
\label{s:local}
For a local estimate we need to define neighborhoods $N_e^{\gamma}$ -- \emph{relator neighborhoods in $\gamma$} -- one for every edge $e$ in $\gamma$, for which the number $|E(N_e^{\gamma})\cap \ag|$ of edges can be bounded from below.

For a given edge $e$ of $\gamma$ we define a corresponding relator neighborhood $N_e^{\gamma}$ as follows.
If $e\in \ag$ then $N_e^{\gamma}=\{ e \}$. Otherwise, we proceed in the way described below.
\medskip

Since $e$ is not in $\ag$, its wall $w$ crosses $\gamma$ in at least one more edge. In the wall $w$, choose an edge $e'\subseteq \gamma$ being  a closest edge-vertex to $e\neq e'$ in the hypergraph $\Gamma_w$ of the wall $w$. 
We consider separately the two following cases.
\medskip 

\noindent
{\bf Case I: The edges $e$ and $e'$ do not lie in a common relator.}
In the hypergraph $\Gamma_w$ of the wall $w$, which is a tree by Theorem~\ref{l:tree}, consider the geodesic $\gamma_w$ between vertices $e$ and $e'$. Let $r_e^{\gamma}$ be the relator-vertex in $\gamma_w$ adjacent to $e$. Let $e''$ be an edge-vertex in $\gamma_w$ adjacent to $r_e^{\gamma}$. Consequently, let $r''$ be the other relator-vertex in $\gamma_w$ adjacent to $e''$.
The intersection of $r_e^{\gamma}$ with $\gamma$ is an interval $p'q'$.
Assume without loss of generality, that $q'$ lies between $e$ and $e'$;  see Figure~\ref{f:C}. 
\medskip

We define the relator neighborhood $N_e^{\gamma}$ as the interval $p'q'=r_e^{\gamma}\cap \gamma$.
The following lemma is the same as \cite[Lemma 4.3]{AO}.
\begin{lemma}
	\label{l:Ne}
	\begin{align*}
	|E(N_e^{\gamma})|> (\beta - \lambda)\, \gi\,r_e^{\gamma}.
	\end{align*}  
\end{lemma}

\medskip 

\noindent
{\bf Case II: The edges \texorpdfstring{$e$}{e} and \texorpdfstring{$e'$}{e'} lie in a common relator \texorpdfstring{$r_e^{\gamma}$}{r}.}
We may assume (exchanging $e'$ if necessary) that $e'$ is closest to $e$ (in $X$) among edges in $w$ lying in $r_e^{\gamma} \cap \gamma$.
\medskip

The relator neighborhood $N_e^{\gamma}$ is now defined as the interval $p'q'=r_e^{\gamma}\cap \gamma$.
By the $\beta$--condition of the $(\beta,\Phi)$--separation property, we have
\begin{align}
\label{e:II1}
|E(N_e^{\gamma})|\geqslant \beta \, \gi  r_e^{\gamma}.
\end{align}

\medskip

In the following two lemmas we estimate the local density of $A(\gamma)$ separately in the two cases.
The lemmas correspond to, respectively, Lemma 4.4 and Lemma 4.5 from \cite{AO}.

\begin{lemma}[Local density of $\ag$ -- Case I]
  \label{l:prop2}
The number of edges in $N_e^{\gamma}$, whose walls separate $p$ from $q$ is estimated as follows:
\begin{align*}
|E(N_e^{\gamma})\cap \ag| \geqslant \Phi((\beta - \lambda)\, \gi r_e^{\gamma})-4\, P(r_e^{\gamma}).
\end{align*}
\end{lemma}
\begin{proof}
  To estimate $|E(N_e^{\gamma})\cap \ag|$ we consider first a set $B$ of edges in $N_e^{\gamma}$ defined in the following way. An edge $f$ belongs to $B$ if its wall $w_f$ has only one edge in common with $N_e^{\gamma}$. In particular, $w_f$ 
separates $p'$ from $q'$.

By the $\Phi$--condition from Definition~\ref{d:sep}, and by Lemma~\ref{l:Ne}, we have
\begin{align}
\label{e:B}
|B|\geqslant \Phi(|E(N_{e})|) \geqslant \Phi((\beta - \lambda)\, \gi  r_e^{\gamma}).
\end{align}

We estimate further the number of edges in $\ag \cap B$. To do this we explore the set of edges $f$ in $B$ 
outside $\ag$.
We consider separately the two ways in which an edge $f$  of $B$ may fail to belong to $\ag$ -- these are studied in Cases: C and D below.

Since $f\in B\setminus \ag $, there exists another edge of the same wall $w_f$ in $\gamma$ outside $r_e^{\gamma}$. Let $f'$ be a closest to $f$ such edge-vertex in
the hypergraph $\Gamma_{w_f}$. Denote by $\gamma_{w_f}$ the geodesic in $\Gamma_{w_f}$ between $f$ and $f'$.
Let $r_f$ be the relator-vertex on $\gamma_{w_f}$ adjacent to $f$.

\begin{figure}[h!]
\centering
\includegraphics[width=0.9\textwidth]{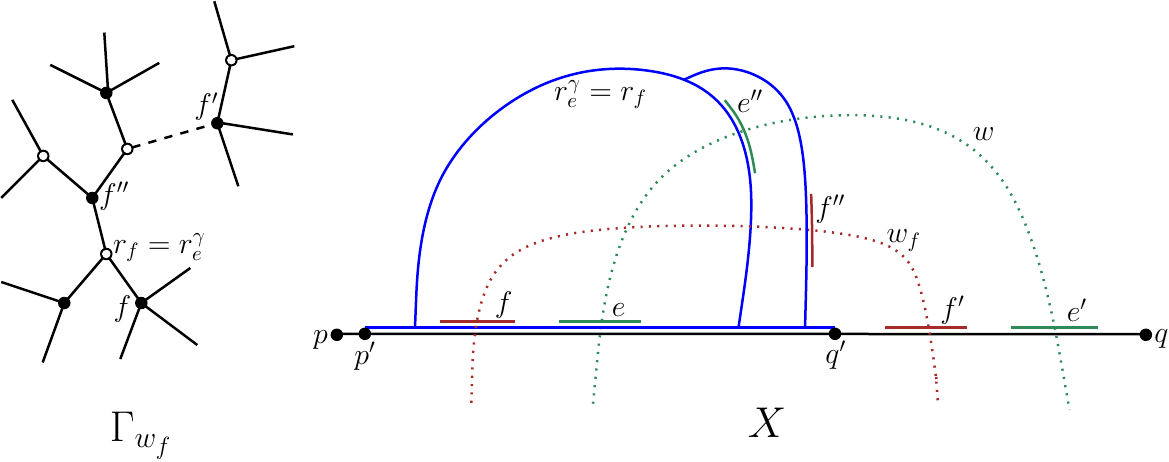}
\caption{Lemma~\ref{l:prop2}, Case I(C).}
\label{f:CC}
\end{figure}

\medskip

\noindent
\emph{(Case C): $r_f=r_e^{\gamma}$.} Observe that then there are at least two distinct relator-vertices between $f$ and $f'$ on $\gamma_{w_f}$; see Figure~\ref{f:CC}.
The cardinality of the set $C$ of such edges $f$ is bounded, by Lemma~\ref{l:C}, as follows:
\begin{align}
\label{e:205}
|C|\leqslant 2\,P(r_e^{\gamma}).
\end{align}

\medskip

\noindent
\emph{(Case D): $r_f\neq r_e^{\gamma}$.}
Let the set of such edges $f$ be denoted by $D$. 
Let $r_f\cap \gamma = p''q''$. 
We claim that $p'\in p''q''$ or $q'\in p''q''$. 
Therefore
\begin{align}
\label{e:290}
|D|\leqslant 2\, P(r_e^{\gamma}).
\end{align}  

To show the claim we proceed by contradiction. Suppose not -- then $p''q''\subseteq p'q'$.
By Lemma~\ref{l:Ne} we have then (treating $r_f$ as $r_f^{\gamma}$)
$|p''q''|> (\beta - \lambda)\, \gi r_f$. However, by our choice of $\beta$, this contradicts the small cancellation condition.

Now we combine the cases C, and D, to obtain the following bound in Case I, see estimates (\ref{e:B}), (\ref{e:205}), and (\ref{e:290}) above.
\begin{align*}
|E(N_e^{\gamma})\cap \ag| &\geqslant |B \cap A(\gamma)|\geqslant  |B|-|C|-|D|\\ & \geqslant 
\Phi((\beta - \lambda)\, \gi r_e^{\gamma})-4\, P(r_e^{\gamma}).
\end{align*}  
\end{proof}

\begin{lemma}[Local density of $\ag$ -- Case II]
  \label{l:prop3}
The number of edges in $N_e^{\gamma}$, whose walls separate $p$ from $q$ is estimated as follows:
\begin{align*}
|E(N_e^{\gamma})\cap \ag| \geqslant \Phi(\beta\, \gi r_e^{\gamma})-4\, P(r_e^{\gamma}).
\end{align*}
\end{lemma}
\begin{proof}
Again, let $B$ be the set of edges $f$ in $N_e^{\gamma}$ such that their wall $w_f$ intersects $N_e^{\gamma}$ in exactly one edge. Then $w_f$ separates $p'$ and $q'$.
As in Case I (see (\ref{e:B})), by (\ref{e:II1}), we have the following lower bound:
\begin{align*}
|B|\geqslant \Phi(|E(N_{e})|) \geqslant \Phi(\beta \, \gi r_e^{\gamma}).
\end{align*}

We estimate again the number of edges $f$ in $B\setminus \ag$. As in Case I (Lemma~\ref{l:prop2}), we consider separately two possibilities: C, D
for such an edge $f$ to fail belonging to $\ag$.
The same considerations as in Case I lead to the estimates:
\begin{align*}
|C|&\leqslant 2\, P(r_e^{\gamma}),\\
|D|&\leqslant 2 \, P(r_e^{\gamma}).
\end{align*}
Combining all the inequalities above we get
\begin{align*}
|E(N_e^{\gamma})\cap \ag| &\geqslant |B \cap A(\gamma)|\geqslant  |B|-|C|-|D|\\ & \geqslant 
\Phi(\beta\, \gi r_e^{\gamma})-4\, P(r_e^{\gamma}).
\end{align*}\end{proof}

We are ready to combine all the previous estimates to obtain the final local estimate.

\begin{lemma}[Local density of $\ag$]
  \label{l:local}
For $e\notin \ag$, the number of edges in $N_e^{\gamma}$ whose walls separate $p$ from $q$ is estimated as follows:
\begin{align*}
|E(N_e^{\gamma})\cap \ag| \geqslant \Omega (\gi r_e^{\gamma}).
\end{align*}
\end{lemma}
\begin{proof}
We use the lacunarity condition from Definition~\ref{d:cond}, and Lemma~\ref{l:prop2} or Lemma~\ref{l:prop3}.
\end{proof}


\subsection{Properness of the wall pseudo-metric}
\label{s:global}
Using the local estimate on the density of $\ag$ from Lemma~\ref{l:local}, we now estimate the overall density of edges with walls separating $p$ and $q$, thus obtaining the properness of the wall pseudo-metric $\dw$.

\begin{theorem}[Properness]
  \label{p:lsp}
There exists a homeomorphism $\Psi \colon [0,+\infty) \to [0,+\infty)$ such that 
  \begin{align*}
    d(p,q)\geqslant \dw(p,q)\geqslant \Psi(d(p,q)).
  \end{align*}
\end{theorem}

\begin{proof}
  The left inequality is clear. Now we prove the right one. Define $\Psi \colon [0,+\infty) \to [0,+\infty)$ as a homeomorphism  such that $\Psi(d) \leqslant \min \{\sqrt{d}/2,\\ \Omega( \Delta(\sqrt d))  \}$. For given $p,q$, we denote $d:=d(p,q)$. If $d=1$ then $1=\dw(p,q)\geqslant \Psi(d(p,q))=\Psi(1)$. Further we assume $d\geqslant 2$.

We work with the family $\{ N_e^{\gamma} \}_{e\subseteq \gamma}$ of relator neighborhoods, as defined in Subsection~\ref{s:local}.
We consider separately the following two cases.
\medskip

\noindent
\emph{(Case 1): There is and edge $e$ in $\gamma$ with $|E(N_e^{\gamma})|\geqslant \sqrt d$}. Observe that
then $e\notin \ag$.
For such an edge $e$, by the large girth condition from Definition~\ref{d:cond}, we have 
\begin{align*}
\gi r_e^{\gamma} \geqslant \Delta(\di r_e^{\gamma}) \geqslant \Delta (|E(N_e^{\gamma})|)\geqslant \Delta (\sqrt d),
\end{align*}
and thus, by Lemma~\ref{l:local}, we obtain
\begin{align}
\label{e:g1}
|\ag|\geqslant |\ag \cap E(N_e^{\gamma})| \geqslant \Omega (\gi r_e^{\gamma}) \geqslant \Omega( \Delta(\sqrt d)).
\end{align}
\medskip

\noindent
{\emph{(Case 2):  For every edge $e$ in $\gamma$ we have $|E(N_e^{\gamma})| < \sqrt d$}}.
Then, as in the proof of 
\cite[Lemma 2.1]{AO0} there is a family $\{e_1,e_2,\ldots, e_k\}$ of edges in $\gamma$, such that $N_{e_i}^{\gamma}\cap N_{e_j}^{\gamma}=\emptyset$, for $i\neq j$ and $k\geqslant \sqrt d/2$. Therefore, by Lemma~\ref{l:local} and by the fact that $|\ag \cap E(N_{e_i}^{\gamma})|=1$ for $e_i\in \ag$, we have
\begin{align}
\label{e:g2}
|\ag| \geqslant \sum_{i=1}^{k} |\ag \cap E(N_{e_i}^{\gamma})| \geqslant \sum_{i=1}^{k} 1 \geqslant \sqrt d/2.
\end{align}
Combining formulas (\ref{e:g1}) and (\ref{e:g2}), we obtain 
\begin{align*}
   \dw(p,q)\geqslant |\ag| \geqslant \Psi(d(p,q)).
  \end{align*}
\end{proof}


\section{PW non-A groups}
\label{s:pwna}

In this section we prove
Theorem 2 from the Introduction; see Theorem~\ref{t:main} below. 
For the whole section we assume that $\Theta$ consists of $D$--regular graphs, for some $D\geqslant 3$. (This assumption could be `coarsely weakened'; see \cite{Willett}.)
We fix $\lambda \in (0,1/24]$ and -- using Theorem~\ref{l:c2LLL} -- a labelling $(\Theta,m)=((\txxT_n,m_n))_{n\in \bN}$ with the following property: no
$m_n$--labelling of a path of length at least $\lambda\, \gi \txT_n$ in  $\txT_n$ appears as the $m$--labelling of some other path in $\txT$.

First, in Subsection~\ref{s:sclab}, we derive from $(\Theta,m)$ an appropriate
 sequence of labelled graphs $(\stxxT,\stxxm)$ -- it consists of coverings of graphs $(\txT_n)$ with the induced labelling. Then, in Subsection~\ref{s:group}, we use the sequence $(\stxxT,\stxxm)$ to define
 a graphical small cancellation group $G$ with the required properties.


\subsection{From \texorpdfstring{$(\txxT,\txxml)$}{(Theta,m)} to \texorpdfstring{$(\hxxT,\hxxm)$}{(tilde Theta,tilde m)} and \texorpdfstring{$(\stxxT,\stxxm)$}{(hat Theta,hat m)}.}
\label{s:sclab}
In this subsection, we define pieces in $(\txxT,m)$ (and $(\hxxT,\hxxm)$, $(\stxxT,\stxxm)$) and $P(\txxT_n)$ 
(and $P(\hxxT)$, $P(\stxxT)$) in the following way, corresponding
to definitions from Subsection~\ref{s:c'cpl} and Section~\ref{s:wlac}. 

Let $p_1\colon p\to \txxT$ be a path in $\txxT$, that is, a locally injective simplicial map from a graph $p$ homeomorphic to a segment.
The path $p_1$ is a \emph{piece} in $(\txxT,m)$ if there exists a different path $p_2\colon p\to \txxT$ inducing the same labelling of $p$. In particular, every piece in $(\txxT_n,m_n)$ has length smaller
than $\lambda\, \gi \txT_n$.
For a vertex $v\in \txxT_n$, by $P_v(\txxT_n)$ we denote
the number of edges of $\txxT_n$ contained in (images of) all pieces containing $v$. Consequently, $P(\txxT_n)$ denotes the maximal number among $P_v(\txxT_n)$, for vertices $v$ of $\txxT_n$.
Observe that for the graphical presentation $\langle \Gamma \; | \; \txxT \rangle$ given by the labelling $m$ the associated complex $X$, as defined in Subsection~\ref{s:c'cpl} has the following property.
The pieces in $X$ (as defined in Subsection~\ref{s:c'cpl}) are exactly the compositions $p\to r_i \stackrel{\varphi_i}{\longrightarrow} X$ where $p\to r_i$ is a piece in a copy $r_i$ of some $\txxT_n$ as defined above. 

Labelled graphs $(\hxxT,\hxxm)$ and $(\stxxT,\stxxm)$ will be defined below as appropriate \emph{coverings of labelled graph} $(\Theta,m)$, that is,
graph coverings with labellings induced from $m$ by the covering map.
A path $p_1\colon p\to \hxxT$ (respectively, $p_1\colon p\to \stxxT$) is a \emph{piece} in $(\hxxT,\hxxm)$ (respectively, $(\stxxT,\stxxm)$) if there is a different path $p_2\colon p\to \hxxT$ (respectively, $p_2\colon p\to \stxxT$) inducing the same labelling of $p$, and such that the following holds.
There is no $n$ and a covering graph automorphism $\alpha\colon \hxxT_n \to \hxxT_n$ (respectively, 
$\alpha\colon \stxxT_n \to \stxxT_n$) that make the following diagrams commutative.
$$  \begindc{\commdiag}[20]
\obj(13,1)[a]{$\hxxT_n$}
\obj(12,1)[a']{}
\obj(35,16)[c]{$\hxxT_n$}
\obj(35,17)[c']{}
\obj(12,16)[d]{$p$}
\obj(12,17)[d']{}
\mor{a'}{c}{$\alpha$}
\mor{d'}{c'}{$p_1$}
\mor{d'}{a'}{$p_2$}[\atright,\solidarrow]
\obj(73,1)[e]{$\stxxT_n$}
\obj(72,1)[e']{}
\obj(95,16)[g]{$\stxxT_n$}
\obj(95,17)[g']{}
\obj(72,16)[h]{$p$}
\obj(72,17)[h']{}
\mor{e'}{g}{$\alpha$}
\mor{h'}{g'}{$p_1$}
\mor{h'}{e'}{$p_2$}[\atright,\solidarrow]
\enddc  $$
The numbers $P(\hxxT_n)$ and $P(\stxxT_n)$ are defined correspondingly. 
Again, pieces in $(\hxxT,\hxxm)$ and $(\stxxT,\stxxm)$ correspond to pieces
in the complexes associated with the graphical presentations $\langle \Gamma \; | \; \hxxT \rangle$ and $\langle \Gamma \; | \; \stxxT \rangle$.

In what follows the labelled graph covering $(\hxxT,\hxxm)$ will be chosen so that $\gi \hxxT_n$ is large compared to $P(\hxxT_n)$; see Lemma~\ref{l:plwg}. We assume that all the coverings $\hxxT_n \to \tn$, $\stxxT_n \to \hxxT_n$, and
$\stxxT_n \to \tn$ are regular (in other words, normal), that is, the corresponding subgroups of fundamental groups are normal. This implies that the groups of covering graph automorphisms act
transitively on fibers. Recall, that the \emph{$\mathbb Z_2$--homology cover} $\widetilde{\Sigma} \to  \Sigma$
is the cover corresponding to the characteristic subgroup of $\pi_1(\Sigma)$ being the kernel of 
the abelianization map $\pi_1(\Sigma) \to H_1(\Sigma; \mathbb{Z}_2)$. The covers $\stxxT_n \to \hxxT_n$ are  
$\mathbb Z_2$--homology covers, and the covers $\hxxT_n \to \tn$ may be (thought of as) iterated 
$\mathbb Z_2$--homology covers. Observe that then all the coverings $\hxxT_n\to \txxT_n$, $\stxxT_n \to \hxxT_n$, and $\stxxT_n \to \txxT_n$ are regular. It is so because characteristic subgroups of normal subgroups
are themselves normal subgroups.

\begin{lemma}
	\label{l:covpiece}
	Every piece in $(\hxxT_n,\hxxm_n)$ (respectively, in $(\stxxT_n,\stxxm_n)$) has length smaller than
	$\lambda \, \gi \txxT_n$. Furthermore, $P(\txxT_n)=P(\hxxT_n)=P(\stxxT_n)$.
\end{lemma}
\begin{proof}
	We treat the case of $(\hxxT_n,\hxxm_n)$ -- the other case can be treated the same way.
	Suppose there is a piece $p_1\colon p\to \hxxT_n$ with $p$ of length at least $\lambda \, \gi \txxT_n$.
	Restricting the domain, we may assume that $|p|<\gi \txxT_n$.
	Then, necessarily, there is a different
	path $p_2\colon p\to \hxxT_n$ inducing the same labelling of $p$ (otherwise we would get too long piece in $(\txxT,m)$). Since the group of covering automorphisms acts transitively on fibers of $\hxxT_n \to \txxT_n$, if there did not exist a covering automorphism $\alpha\colon \hxxT_n \to \hxxT_n$
	such that the diagram
$$  \begindc{\commdiag}[20]
\obj(13,1)[a]{$\hxxT_n$}
\obj(12,1)[a']{}
\obj(35,16)[c]{$\hxxT_n$}
\obj(35,17)[c']{}
\obj(12,16)[d]{$p$}
\obj(12,17)[d']{}
\mor{a'}{c}{$\alpha$}
\mor{d'}{c'}{$p_1$}
\mor{d'}{a'}{$p_2$}[\atright,\solidarrow]
\enddc  $$
commutes, then the compositions of $p_1,p_2$ with the covering map $\hxxT_n \to \txxT_n$ would result
in different paths in $\txxT_n$ inducing the same labelling of $p$. This would lead to contradiction.
Hence such $\alpha$ exists for every $p_2$ and it follows that $p_1$ is not a piece -- contradiction.

The second statement follows from the fact that the union of (images of) all the pieces containing a given vertex $v$ in $\hxxT_n$ is mapped isometrically onto the union of (images of) all the pieces containing
the image of $v$ in $\tn$.
\end{proof}

For the $\mathbb Z_2$--homology cover $\stxxT_n \to \hxxT_n$, as observed by Wise (see \cite[Section 9]{W-qch} and \cite[Section 10.3]{W-ln}), every $\stxxT_n$ is equipped with a structure of graph with walls -- a wall consists of edges in $\stxxT_n$ being preimages
of a given edge in $\hxxT_n$ (see also \cite[Section 3]{AGS} and \cite[Lemma 6]{Ost}). 
With this system of walls we obtain the following lemma, which will allow us to conclude that the $C'(\lambda)$--complex associated with the graphical presentation $\langle \Gamma \, |\, \stxxT \rangle$
satisfies the proper lacunary walling condition from Definition~\ref{d:cond}.
\begin{lemma}
\label{l:plwg}
There exist coverings $(\hxxT_n,\hxxm_n) \to (\txxT_n,m_n)$ of appropriately large girth such that the following holds.
There exist: $\beta\in (0,1/2]$, and homeomorphisms $\Phi,\Omega,\Delta\colon [0,+\infty) \to [0,+\infty)$
such that, for every $n\in \bN$ we have:
\begin{enumerate}
\item
the degree of $\stxxT_n$ is bounded by $D$;
\item
the length of each piece in $\stxxT_n$ is at most
$\lambda \, \gi \stxxT_n$;
\item
$\stxxT_n$ satisfies the $(\beta,\Phi)$--separation property;
\item
$\Phi((\beta - \lambda) \gi \, \stxxT_n)-4\, P(\stxxT_n) \geqslant \Omega(\gi \stxxT_n)$;
\item
$\gi \stxxT_n \geqslant \Delta(\di \stxxT_n)$. 
\end{enumerate}
\end{lemma}
\begin{proof}
(1) is immediate. (2) follows from Lemma~\ref{l:covpiece}. The existence of $\Delta$ satisfying 
(5) follows from the fact that $\gi \stxxT_n \to \infty$ as $n\to \infty$.

For (3), the $\beta$-condition from Definition~\ref{d:sep} holds with $\beta=1/2$, by \cite[Lemma 7.1]{AO}.
Now, we show how to choose $\Phi \colon [0,+\infty)\to [0,+\infty)$ such that the $\Phi$-condition holds.
First, we choose inductively the coverings $(\hxxT_n,\hxxm_n) \to (\txxT_n,m_n)$ so that the following condition $({\ast})$
is satisfied, for every $n$:
\begin{center}
	 $(\gi \hxxT_n)/3 -4\, P(\stxxT_n)>(\gi \hxxT_n)/4$, and \\$({\ast})\; \;$ there does not exist a geodesic of length at least $(\gi \hxxT_n)/3$ \\in $\stxxT_j$, for $j<n$.
\end{center}
Such a choice is obviously possible because geodesics are simple paths, the graphs $\hxxT_j$ are finite, and
$P(\txxT_n)=P(\hxxT_n)=P(\stxxT_n)$, by Lemma~\ref{l:covpiece}.
Now, for a given number $N\in \mathbb{N}$, we define a number $\widetilde{\Phi}(N)$ as follows:
$\widetilde{\Phi}(0)=0$, and for $N>0$ we find a maximal $n$ such that $N\geqslant (\gi \hxxT_n)/3$, and
we set $\widetilde{\Phi}(N):=\min \{ N,\gi \hxxT_n\}$. Consider a geodesic $\gamma$ of length $N$ in some
$\stxxT_j$. By the condition $(\ast)$, we have $j\geqslant n$, for $n$ as above.
Let $\widetilde{\gamma}$ be the image of $\gamma$ by the projection $\stxxT_j\to \hxxT_j$.
Then $\widetilde{\gamma}$ is an \emph{admissible path} in $\hxxT_j$ in the sense of \cite[Definition 3.5]{AGS}, and the edge-length
of $\widetilde{\gamma}$ is $N$ as well \cite[Lemma 3.6 and Proposition 3.8]{AGS}. Since $\gamma$ is a geodesic,
the path $\widetilde{\gamma}$ has no backtracks \cite[Remark 3.9]{AGS}. Hence, if $\widetilde{\gamma}$ does not contain any loop, then every edge in $\widetilde{\gamma}$ is traversed only once, and consequently,
the number of edges in $\gamma$ whose walls have exactly one edge in common with $\gamma$ is $N$.
If $\widetilde{\gamma}$ contains a loop then, necessarily, the length of this loop is at least
$\gi \hxxT_n$. By \cite[Lemma 3.12]{AGS}, every edge on the loop is traversed exactly once, so the number 
 of edges in $\gamma$ whose walls have exactly one edge in common with $\gamma$ is at least $\gi \hxxT_n$.
Combining the two cases, we get that the number of edges in $\gamma$ whose walls have exactly one edge in common with $\gamma$ is at least $\widetilde{\Phi}(N)$. Therefore, there exists a homeomorphism 
$\Phi \colon [0,+\infty) \to [0,+\infty)$ satisfying $\Phi(x)\geqslant \widetilde{\Phi}(x)$, for all $x\in  [0,+\infty)$, and ensuring the $\Phi$-condition. Observe that, additionally, $\Phi$ can be chosen so that $\Phi((\gi \hxxT_n)/3)\geqslant (\gi \hxxT_n)/3$, for all $n$.

Since $\beta=1/2$ and $\lambda \leqslant 1/24$, by the choice of $\Phi$, we have
\begin{align*}
\Phi((\beta - \lambda) \gi \, \stxxT_n)>\Phi((\gi \, \hxxT_n)/3)\geqslant(\gi \, \hxxT_n)/3.
\end{align*}
Hence, by the condition $(\ast)$, there exists $\Omega$ as in (4).
\end{proof}


\subsection{The group}
\label{s:group}
Now we construct a coarsely non-amenable group acting properly on a CAT(0) cubical complex announced in Theorem 2. 
The group is defined by a graphical small cancellation presentation over the sequence $\stxxT$; see Section~\ref{s:grexp} for notations. 
Again, $\Gamma$ is a bouquet of loops, and the local isometries $\varphi_n\colon \stxxT_n \to \Gamma$ are defined
by the labellings $\stxxm_n$. We use $(\stxxT,\stxxm)$ from Lemma~\ref{l:plwg}, that is, satisfying the conditions (1)-(5) there.

\begin{theorem}[PW non-A group]
\label{t:main}
Let $G$ be the group defined by the graphical presentation $\langle \Gamma \, |\, \stxxT \rangle$, where the local isometries
$\stxxT_n \to \Gamma$ are defined by labellings $\stxxm_n$. Then $G$ acts properly on a CAT(0) cubical complex and $G$ does not have property A. 
\end{theorem}
\begin{proof}
Let $X$ be the complex associated to the graphical presentation $\langle \Gamma \, |\, \stxxT \rangle$, as defined in Subsection~\ref{s:c'cpl}.	
By Lemma~\ref{l:plwg}(2), $X$ is a $C'(\lambda)$--complex, where relators $(r_i)$ are copies of graphs $(\stxxT_n)$.
	

Therefore, by Lemma~\ref{l:remb}, the graphs $\stxxT_n$ embed isometrically into the Cayley graph $X^{(1)}$ of $G$. Since $\stxxT_n$ are regular of degree $D\geqslant 3$ and with girths tending to infinity, by a result of Willett \cite{Willett}, the graph $X^{(1)}$ and, consequently, $G$ have no property A.

To show that $G$ acts properly on a CAT(0) cubical complex it is enough \cites{Nica,ChaNib} to show that $G$ acts properly (with respect to the wall pseudo-metric) on a space with walls. Clearly $G$ acts properly on $X^{(0)}$ and thus it remains to show that $X$ satisfies the proper lacunary walling condition to conclude, from Theorem~\ref{p:lsp}, that $G$ acts properly on $(X^{(0)},\mathcal W)$. The proper lacunary walling condition follows from Lemma~\ref{l:plwg}: separation follows from (3), lacunarity from (4), and the large girth condition follows from (5).
\end{proof}


\end{document}